\documentclass[11pt]{article}

\usepackage{graphicx}
\usepackage{tikz}
\usetikzlibrary{trees,shapes,arrows}

\usepackage{paralist}
\usepackage{subfig}
\usepackage{sidecap}
\usepackage{array}

\usepackage{amsthm}
\usepackage{amsmath}
\usepackage{amsfonts}
\usepackage{mathtools}

\usepackage{float}
\usepackage{hyperref}
\hypersetup{
	colorlinks=false,
	linkbordercolor=gray,
	citebordercolor=gray,
	urlbordercolor=gray}%
\usepackage{cleveref}

\usepackage{sectsty}
\allsectionsfont{\sffamily\mdseries\scshape\centering\nohang}

\usepackage{abstract}

\usepackage{caption}
\usepackage{sidecap}

\allowdisplaybreaks

\usepackage[switch,modulo]{lineno}
\usepackage{amssymb}
\usepackage[affil-it]{authblk}

\newtheoremstyle{mplain}
{\topsep}   % ABOVESPACE
{\topsep}   % BELOWSPACE
{\itshape}  % BODYFONT
{0pt}       % INDENT (empty value is the same as 0pt)
{\scshape} % HEADFONT
{:}         % HEADPUNCT
{5pt plus 1pt minus 1pt} % HEADSPACE
{}          % CUSTOM-HEAD-SPEC

\renewenvironment{proof}[1][\proofname]{{\noindent \scshape #1. }}{\qed}

\theoremstyle{mplain}

\newtheorem{theorem}{Theorem}[]
\crefname{theorem}{theorem}{theorems} 
\Crefname{theorem}{Theorem}{Theorems}

\newtheorem{lemma}[theorem]{Lemma}
\crefname{lemma}{lemma}{lemmas} 
\Crefname{lemma}{Lemma}{Lemmas}

\newtheorem{proposition}[theorem]{Proposition}
\crefname{proposition}{proposition}{propositions} 
\Crefname{proposition}{Proposition}{Propositions}

\newtheorem{cor}[theorem]{Corollary}
\crefname{cor}{corollary}{corollaries} 
\Crefname{cor}{Corollary}{Corollaries}

\crefname{definition}{definition}{definitions} 
\Crefname{definition}{Definition}{Definitions}

\newtheorem{example}[theorem]{Example}
\crefname{example}{example}{examples} 
\Crefname{example}{Example}{Examples}

\crefname{remark}{remark}{remarks} 
\Crefname{remark}{Remark}{Remarks}

\crefname{conj}{conjecture}{conjectures} 
\Crefname{conj}{Conjecture}{Conjectures}

\newtheorem{problem}[theorem]{Problem}
\crefname{problem}{problem}{problems} 
\Crefname{problem}{Problem}{Problems}

\newtheorem{claim}{Claim}[]
\crefname{claim}{claim}{claims} 
\Crefname{claim}{Claim}{Claims}

\numberwithin{theorem}{section}

\newcommand{\avq}{\bar q}
\newcommand{\E}{\mathbb{E}}

\renewcommand{\Pr}{\mathbb{P}}

\newcommand{\R}{\mathbb{R}}

\DeclareMathOperator{\Geom}{Geom}

\DeclareMathOperator{\Var}{Var}
\providecommand{\keywords}[1]{\textbf{\textit{Keywords---}} {\small #1}}

	\title{Speed and concentration of the covering time for structured coupon collectors}
	\author{Victor Falgas-Ravry
		\thanks{Electronic address: \texttt{victor.falgas-ravry@vanderbilt.edu}. Research partially supported by a grant from the Kempe foundation.}}
	\affil{Department of Mathematics, Vanderbilt University}
	\author{Joel Larsson
		\thanks{Electronic address: \texttt{joel.larsson@math.umu.se}}}
	\affil{Institutionen f\"or Matematik och Matematisk Statistik, Ume{\aa} Universitet}
	\author{Klas Markstr{\"o}m
		\thanks{Electronic address: \texttt{klas.markström@math.umu.se}. Research supported by a grant from the Swedish Research Council (Vetenskapsr{\aa}det)}
	\affil{Institutionen f\"or Matematik och Matematisk Statistik, Ume{\aa} Universitet}}
\begin{document}
\maketitle

%-----------------------------------------------------------------------------------------------------------------------------------------------------------------------------------------------------------------------------
\begin{abstract}
Let $V$ be an $n$-set, and let $X$ be a random variable taking values in the powerset of $V$. Suppose we are given a sequence of random coupons $X_1, X_2, \ldots $, where the $X_i$ are independent random variables with distribution given by $X$. The covering time $T$ is the smallest integer $t\geq 0$ such that $\bigcup_{i=1}^tX_i=V$. The distribution of $T$ is important in many applications in combinatorial probability, and has been extensively studied. However the literature has focussed almost exclusively on the case where $X$ is assumed to be symmetric and/or uniform in some way. 

In this paper we study the covering time for much more general random variables $X$; we give general criteria for $T$ being sharply concentrated around its mean, precise tools to estimate that mean, as well as examples where $T$ fails to be concentrated and when structural properties in the distribution of $X$ allow for a very different behaviour of $T$ relative to the symmetric/uniform case.

\end{abstract}

%-----------------------------------------------------------------------------------------------------------------------------------------------------------------------------------------------------------------------------
\section{Introduction}\label{section: introduction}
In this paper, we study the random covering problem in a general setting: we are interested in the distribution of the covering time $T$ for general distributions of the random covering variable $X$.  With the exception of  a result of Aldous~\cite{Aldous91}, discussed later, this is as far as we are aware the first time the covering problem is studied in this generality. However the question is a natural one:  there are many applications where the covering variable is §non-uniform' in a way which puts it outside the current literature on covering problems. Also, a common drawback of many of the existing exact results about covering processes is that the expressions obtained often involve a large number of summands and are hard to evaluate directly; this was pointed out for example by Sellke~\cite{Sellke95} and Adler and Ross~\cite{AdlerRoss01}.

Our own focus is on simple, easy-to-use concentration inequalities for the covering time which can be applied in a straightforward way. The the basic question we seek to answer: how does the distribution of $X$ affect the covering time? Can we exploit `structure' in the choice of $X$ to `speed up' or `slow down' the covering? And when can we guarantee that $T$ is sharply concentrated?

Our paper is structure as follows. In Section~\ref{section: preliminaries}, we gather together elementary bounds for the covering time, and identify the range of possible speeds of the covering process, giving examples going over the entire spectrum. We follow on in Section~\ref{section: slow coverage} with the main results of this paper, namely general structure theorems giving sufficient conditions for the covering time of an arbitrary random covering variable to be sharply concentrated. These are stated in Section~\ref{subsection: results} and proved in Sections~\ref{subsection: proofs: concentration}--\ref{subsection: proofs: exchangeable coupons}. In Section~\ref{section: fast coverage} we discuss `fast' coverage by structured random variables. Finally in Section~\ref{section: applications} we give some applications of our results to the connectivity of random graphs, continuum percolation, random graph colourings, the unsatisfiability threshold for $k$-SAT and the appearance of perfect matchings in random graphs. We end with some questions and remarks.

%-------------------------------------------------------------------------------------------------------
\subsection{Definitions}
Let $V$ be a finite set; usually we shall take $V=[n]:=\{1,2, \ldots, n\}$. Let $X$ be a random variable taking values in the power-set of $V$. A random variable $X$ taking values in the power-set of $V$ is referred to as a \emph{random covering variable}, or \emph{random coupon}. We say that $X$ is \emph{exchangeable} if the law of $X$ is invariant under every permutation of $V$. We call $X$ \emph{transitive} if the law of $X$ is invariant under the action of a transitive subgroup of $\mathrm{Sym}(V)$. We say $X$ is \emph{balanced} if for every $v,v' \in V$ we have $\mathbb{P}(v\in X)= \mathbb{P}(v' \in X)$. Finally, $X$ is \emph{$k$-uniform}  if $\vert X\vert =k$ with probability $1$.

We consider an infinite sequence $\mathbf{X} =\{X_1, X_2, \ldots \}$ of i.i.d. random covering variables $X_i \sim X$. We view this as a sequence of random coupons received by a coupon collector; we refer to $X_i$ as the $i^{\textrm{th}}$ coupon, and to the collector as the $X$-coupon collector. We set $C_t=C_t(\mathbf{X}) =\bigcup_{i \leq t} X_i$ to be the collection of elements of $V$ covered by the union of the first $t$ coupons $X_1, X_2, \ldots, X_t$, and define the \emph{covering time} $T=T(\mathbf{X})$ to be 
\[T= \inf\{t: \ C_t =V\}.\]

This quantity $T$ is sometimes referred to as the \emph{waiting time} in the literature. Note that $T$ could be infinite if, for example, $X$ almost surely does not cover (contain) some element $v \in V$. We also define
\[T_{\frac{1}{2}}=T_{\frac{1}{2}}(X) =\inf\left\{t: \ \mathbb{P}(C_t =V)\geq \frac{1}{2}\right\},\]
to be the earliest time by which we have at least a fifty percent chance of having covered $V$, and for a subset $A\subseteq V$ we let $\tau_A=\tau_A(\mathbf{X})$ be the least $t$ such that $A\subseteq C_t$ if it exists, and infinite otherwise. For $v\in V$, we let $d_v(t)$, the \emph{degree} of $v$ at time $t$, denote the number of sets $X_i$ with $i \leq t$ containing $v$.

Our aim in this paper is to prove concentration results for the covering time $T$ in a general setting, i.e. for arbitrary random covering variables $X$. We shall also consider applications where $V\subseteq \mathbb{R}^d$ is a compact set and $X$ takes values among the compact subsets of $V$, and define $V_t$, $T$ and $T_{\frac{1}{2}}$ analogously to the discrete case. In this continuous setting, we shall use $\vert A\vert$ to denote the Lebesgue measure of a set $A$. For a sequence of events $(\mathcal{A}_n)_{n\in \mathbb{N}}$, we say that $\mathcal{A}_n$ holds \emph{with high probability} (\emph{whp}) if
\[\lim_{n \rightarrow \infty} \mathbb{P}(\mathcal{A}_n)=1.\] 
Also, we say that a sequence of random variables $(Y_n)_{n\in \mathbb{N}}$ is \emph{sharply concentrated} around $f(n)$ if $\left({Y_n}/{f(n)}\right)_{n\in \mathbb{N}}$ converges to $1$ in probability, i.e. $\forall \varepsilon>0$, $\lim_{n\rightarrow\infty} \mathbb{P}(\vert Y_n/f(n)-1\vert>\varepsilon)=0$. We recall here the standard Landau notation for asymptotic behaviour. For functions $f,g: \ \mathbb{N}\rightarrow \mathbb{R}_{\geq 0}$, we say that $f= O(g)$ if there exists $C>0$ such that $f(n) \leq Cg(n)$ for all but finitely many $n$. We write $f=o(g)$ to denote that $\lim_{n\rightarrow \infty} f(n)/g(n)=0$. Finally, we use $f=\Omega(g)$ to denote $g=O(f)$, we write $f=\theta(g)$ if both $f=O(g)$ and $f=\Omega(g)$ hold, and use $f=\omega(g)$ or $f\gg g$ to denote $g=o(f)$.

%-------------------------------------------------------------------------------------------------------
\subsection{Some examples}
We give below some examples of random covering variables $X$, illustrating the definitions of exchangeable, transitive and balanced above. 

Our first example is that of the quintessential `nice' random covering variable: the $k$-uniform, exchangeable random coupon variable, which was the focus of most of the previous work on coupon collecting.
\begin{example}\label{example: k-uniform exchangeable}
	Let $X$ be a $k$-set of $V=[n]$ selected uniformly at random, for some $k: \ 1\leq k \leq n$; $X$ is $k$-uniform and exchangeable.
\end{example}

Next we give three examples of `structured' coupon collectors, of the kind that motivate our work in this paper.
\begin{example}
	Let $G$ be a graph on $n$ vertices. Let $X$ be the random coupon obtained by selecting a vertex $x$ of $V=V(G)$ uniformly at random and taking as the coupon 
	the closed neighbourhood of $x$ in $G$, $\bar{\Gamma}(x):=\{y \in V(G): \ xy \in E(G)\}\cup\{x\}$. Here $X$ is balanced if and only if the graph $G$ is regular, and 
	transitive if and only if the graph $G$ has a transitive automorphism group.	
\end{example}
\begin{example}
	Let $V=Q_d$ be the discrete $d$-dimensional hypercube $\{0,1\}^d$, and let $X$ be a $k$-dimensional subcube of $Q_d$ chosen uniformly at random, for some 
	$k: \ 0\leq k \leq d$. This random covering variable $X$ is transitive and $2^k$-uniform but not exchangeable, and, as described in Section \ref{section: applications}, 
	underlies the random SAT problem. 
\end{example}
\begin{example}\label{example: covering a square with discs}
	Let $V$ be the square of area $n$, $[0, \sqrt{n}]^2\subset \mathbb{R}^2$, and let $X$ be the intersection of $V$ with the disc of radius $r$ about a uniformly chosen 
	random point $x\in V$. This random covering variable $X$ is neither uniform nor balanced, due to boundary effects; it is relevant to problems of coverage in random 
	geometric graph theory (see Section~\ref{section: applications}).
\end{example}

%-------------------------------------------------------------------------------------------------------
\subsection{Motivation for coupon collecting}
The problem of determining the covering time of a set by a union of random subsets is of fundamental importance in several areas of mathematics, most notably in probability theory, discrete mathematics and mathematical statistics. This importance is illustrated both by the age of the problem --- in its simplest form, the covering problem can be traced back to de Moivre~\cite{Moivre11} in 1711 --- and by the many appellations it has amassed through the years. It has been studied by a large number of mathematicians from a variety of backgrounds and under a variety of names: matrix occupancy~\cite{EickerSiddiquiMielke72}, allocation of particles in complexes~\cite{VatutinMikhailov83}, committee problem~\cite{MantelPasternack68}, chromosome problem~\cite{Sprott69}, urn-sampling~\cite{Sellke95} or urn-occupancy problem~\cite{ErdosRenyi61}, the Dixie cup problem~\cite{NewmanShepp60}, and, perhaps most famously, the coupon collector problem~\cite{BaumBillingsey65}.

The ubiquitous nature of the covering problem is due to its wide range of applications. It is linked to the study of random walks~\cite{Aldous89}, colouring~\cite{Chvatal91} and degree sequences~\cite{McKaySkerman13} in graph theory. In Section~\ref{section: applications} we also give applications of the coupon collector problem to the connectivity of random graphs. The performance analysis of many exploration or optimisation algorithms in theoretical computer science involves a solution to a covering problem~\cite{PapanicolaouKokolakisBoneh98, VasudevanTowsleyGoeckelKhalili09}, while the unsatisfiability threshold for SAT corresponds to the cover time of a hypercube by random subcubes~\cite{KaporisLefterisStamatiouVamvakari01, FalgasRavryLarssonMarkstrom14}. The `reverse' coupon collector problem --- estimating the size of $V$ given $C_t$ ---
is important to IP traceback algorithms~\cite{SavageWetherallKarlinAnderson01} and the study of biological diversity~\cite{PatilTaillie77, NealMoriary09} amongst other applicationss, while the study of the degrees $d_v(t)$, $v\in V$, is central to hashing and load balancing~\cite{RaabSteger98}. There are further applications in population genetics~\cite{Kingman78,PoonDavisChao05}, evolutionary algorithms for fitness selection~\cite{Poli05} and disordered system physics~\cite{Huillet03}.

%-------------------------------------------------------------------------------------------------------
\subsection{Previous work on coupon collecting}
Most of the previous work on covering problems in the spirit of the present paper focussed on the case where $X$ is an exchangeable, $k$-uniform random covering variable. The case $k=1$, known as the \emph{coupon collector's problem} has received by far the most attention. It can be traced back to de Moivre~\cite{Moivre11}, who computed the probability that $\mathbb{P}(V_t=V)$ exactly. Laplace~\cite{Laplace74} later generalised de Moivre's result to the $k$-uniform case for $k\geq 1$.

The second half of the twentieth century saw great activity on the problem, with many results replicated independently by researchers. P\'olya~\cite{Polya30} gave an expression for the expected covering time $T$ in the $k$-uniform exchangeable case. Feller's textbook~\cite{Feller50} included a computation of $\mathbb{E}T$ in the special case $k=1$. Still in the case $k=1$, Newman and Shepp~\cite{NewmanShepp60} computed the expected time necessary for $m$-coverage of $V$ (covering every point at least $m$ times). Erd\H{o}s and R\'enyi~\cite{ErdosRenyi61} computed the asymptotic distribution of the $m$-coverage time for $m\geq 1$; as their result is of particular relevance to this paper, we state it below:
\begin{theorem}[Erd{\H o}s--R\'enyi]\label{theorem: erdos renyi}
Let $V=[n]$, and let $X$ be the random coupon obtained by selecting a singleton from $V$ uniformly at random. Denote by $T^m$ be the time at which every point of $V$ has been covered by at least $m$ of the coupons $X_1, \ldots , X_{T^m}$. Then for every $x\in\mathbb{R}$, 
\[\lim_{n\rightarrow \infty}\Pr\left(T^m<n \log n + (m-1)n\log \log n  +xn \right)=e^{-e^{-x}}.\]
In particular $T^m$ is sharply concentrated around $n \log n + (m-1)n\log \log n$. 
\end{theorem}
Continuing work on the $1$-uniform exchangeable case, Baum and Billingsley~\cite{BaumBillingsey65} proved results on the asymptotic distribution of the size of $C_t$ (the number of coupons collected by time $t$)as a function of $t$; Holst~\cite{Holst77} later generalised their result to unbalanced $1$-uniform random covering variables $X$. A number of researchers worked on the distribution of the degrees $(d_v(t))_{v\in V}$, such as Eicker, Siddiqui and Mielke~\cite{EickerSiddiquiMielke72}, Mikhailov~\cite{Mikhailov78}, Barbour and Holst~\cite{BarbourHolst89} and Khakimullin and Enatskaya~\cite{KhakimullinEnatskaya97}, all of whom dealt with the $k$-uniform case with $k> 1$ as well. A number of the papers cited above also deal with the unbalanced, $1$-uniform case; let us mention in addition the work of Papanicolaou, Kokolakis and Boneh~\cite{PapanicolaouKokolakisBoneh98}, who gave an expression for the expected covering time when $X$ is a randomly chosen $1$-uniform random covering variable.

In the exchangeable $k$-uniform case with $k>1$, several researchers~\cite{MantelPasternack68,Gittelsohn69,Stadje90} computed, like Laplace, the expected covering time, giving closed-form formulae. Vatutin and Mikhailov~\cite{VatutinMikhailov83} determined the asymptotic distribution of the number of degree zero (i.e. uncovered) vertices, which in turn gives results on the distribution of the covering time.

Recently Ferrante and Frigo~\cite{FerranteFrigo12} gave an expression for the expected covering time when $X$ is a $k$-uniform covering random variable with different $v\in V$ receiving different weights. In a different direction, improving results of Sellke, Ivchenko~\cite{Ivchenko98} computed the asymptotic distribution of the covering time when $n\rightarrow \infty$  and $X$ is a fixed (i.e. not varying on $n$) non-uniform exchangeable random variable; similar results were also obtained by Johnson and Sellke~\cite{JohnsonSellke10}, while a closed-form expression for the expectation of $T$ appeared in Adler and Ross~\cite{AdlerRoss01}.

Finally, Aldous~\cite{Aldous91} proved a general abstract result about covering times, in connection with random walks on graphs. To state his result, we need one more definition. Given a random covering variable $X$ and an $X$-coupon collector, we let $B=B(\mathbf{X})$ denote the set of ``holdouts'', which is to say the last subset of $V$ to be covered: $B=C_{T}\setminus C_{T-1}$ (if the coupon collector does not cover $V$, we set $B$ to be the collection of never-covered elements of $V$).
\begin{theorem}[Aldous]\label{theorem: aldous}
	Suppose $\E T=\omega(1)$. Then $T$ is sharply concentrated around its expectation if and only if
	\[\frac{\E_B (\E \tau_B) }{\E T}=o(1).\]
\end{theorem}
The power of Aldous's theorem is its generality and the necessity and sufficiency of its hypothesis for the concentration of the covering time. However as Aldous observed ``[w]ithout any structure being imposed [...] it is not clear how to estimate [$\E_B (\E \tau_B)$] in order to use these results''. Indeed, computing
\[\E_B (\E \tau_B)= \sum_{A\subseteq V}  \Pr(A=B) \E \tau_A\]
requires us to estimate both the probability that a given set $A$ is the ``holdout'' and to compute its expected covering time, both of which may be non-trivial tasks.

Several surveys have been written on coupon collectors, random allocation, urn occupancy problems, etc. Amongst others, let us mention the book of Johnson and Kotz~\cite{JohnsonKotz77} and Kolchin, Sevast'yanov and Chistyakov~\cite{KolchinSevastyanovChistyakov78}, the surveys of Ivanov, Ivchenko and Medvedev \cite{IvanovIvchenkoMedvedev85} and Kobza, Jacobson and Vaughan~\cite{KobzaJacobsonVaughan07}, and the papers of Holst~\cite{Holst86}, Stadje~\cite{Stadje90}, Flajolet, Gardy and Thimonier~\cite{FlajoletGardyThimonier92} and McKay and Skerman~\cite{McKaySkerman13}.

%-----------------------------------------------------------------------------------------------------------------------------------------------------------------------------------------------------------------------------
\section{Preliminaries: thresholds and elementary bounds}\label{section: preliminaries}

%-------------------------------------------------------------------------------------------------------
\subsection{Coarse threshold}

It follows from a simple application of the Bollob\'as--Thomason threshold theorem~\cite{BollobasThomason87} that a covering process as we have defined it will always have a coarse threshold:
\begin{proposition}[Coarse threshold]\label{cor: Bollobas-Thomason concentration}
	Let $X$ be a covering random variable for a set $V$. Then 
	\begin{align*}
		\Pr(C_t=V) &=\left\{\begin{array}{ll}
		o(1) & \textrm{if $t\ll T_{\frac{1}{2}}$}\\
		1-o(1) & \textrm{if $t\gg T_{\frac{1}{2}}$.}
		\end{array} \right. &&\qed
	\end{align*}
\end{proposition}
\noindent Thus the covering time $T$ is whp of the same order as $T_{\frac{1}{2}}$.  In the present work, however, we are interested in a much sharper form of concentration than the one guaranteed by Proposition~\ref{cor: Bollobas-Thomason concentration}: we want the covering time $T$ to be \emph{sharply concentrated}, i.e. we want that  $T/T_{\frac{1}{2}} \rightarrow 1$ in probability. As we shall see in the next subsection, we cannot in general guarantee this kind of sharp concentration. A question of crucial interest is then what conditions 
are necessary or sufficient to have sharp concentration for $T$ --- and how the value of $\mathbb{E}T$ may be computed in such cases.

%-------------------------------------------------------------------------------------------------------
\subsection{Elementary bounds}
Let $X$ be a covering random variable for an $n$-set set $V$. For each $v\in V$, let $q_v=\mathbb{P}(v\in X)$, and set $q_{\star}=\min_v q_v$. We have the following elementary bounds on the location of $T_{\frac{1}{2}}$ and probable location of $T$.
\begin{proposition}\label{proposition: elementary bounds}
	\[  \frac{\log (2)}{-\log (1-q_{\star})}  \leq T_{\frac{1}{2}} \leq \frac{\log (2n)}{-\log (1-q_{\star})}.\]
	What is more, 
	for any fixed $\varepsilon>0$
	\[ \mathbb{P}\left(T\leq (1+\varepsilon) \frac{\log n}{-\log\left(1-q_{\star}\right)}\right) \geq 1-n^{-\varepsilon}.\]
\end{proposition}

\begin{proof}
For $t\geq T_{\frac{1}{2}}$ we have
\begin{align*}
\frac{1}{2}\leq \mathbb{P}(C_t=V)&\leq \inf_{v\in V}\mathbb{P}(v \in C_t)=1-\left(1- q_{\star}\right)^t,
\end{align*}
from which the claimed lower bound on $T_{\frac{1}{2}}$ follows. For the upper bound, $t\leq T_{\frac{1}{2}}$ implies
\begin{align*}
\frac{1}{2}\leq \mathbb{P}(C_t\neq V)\leq \sum_{v\in V}\mathbb{P}(v \notin C_t)&=\sum_{v\in V}(1-q_v)^t \leq n (1-q_{\star})^t.
\end{align*}
Finally, for the `what is more' statement, note that for $t\geq (1+\varepsilon) \frac{\log n}{-\log \left(1-q_{\star}\right)}$, the expected number of vertices not yet collected is
\begin{align*}
\mathbb{E}\vert V\setminus C_t \vert &=n (1-q_{\star})^t\leq n^{-\varepsilon},
\end{align*}
whence by Markov's inequality with probability at least $1-n^{-\varepsilon}$ we have $C_t=V$ and $T\leq t$.
\end{proof}

Note that if $X$ is balanced then $q_v=\frac{\mu}{n}$ for all $v\in V$, where $\mu:=\mathbb{E} \vert X\vert$. In particular if $\mu=o(n)$ then the bounds above can be rewritten as 
\[(1+o(1))\frac{n\log 2}{\mu}= \frac{\log 2}{-\log \left(1-\frac{\mu}{n}\right)} \leq T_{\frac{1}{2}}\leq \frac{\log (2n)}{-\log \left(1-\frac{\mu}{n}\right)} = (1+o(1))\frac{n\log n}{\mu}.\]

Perhaps surprisingly, these elementary bounds are essentially sharp. As we shall show in the next section, the covering time $T$ for the exchangeable $k$-uniform coupon collector (Example~\ref{example: k-uniform exchangeable}) is sharply concentrated around the value $\frac{\log n}{-\log \left(1-\frac{k}{n}\right)}$; in particular if $k=o(n)$, $T_{\frac{1}{2}}=(1+o(1))\frac{n \log n}{k}$. We think of this as `slow coverage'. On the other hand, there are instances of `fast coverage', discussed in greater detail in Section~\ref{section: fast coverage}. We give here a simple example.

\begin{example}[Coupon collector with lottery]\label{example: lottery}
	Set $V=[n]$, 
	and $p=c/n$  for some $c=c(n)\in [0,n]$. Let $X$ be with probability $1-p$ 
	a singleton from $V$ chosen uniformly at random, and with probability $p$ the entire set $V$. 
\end{example}
\noindent Note that $X$ is an exchangeable random covering variable, with expected size
\[\mathbb{E}\vert X\vert = (1-p)+pn=1 +c -\frac{c}{n}.\] 
\begin{proposition}\label{prop: lottery covering}
	Let $X$ and $V$ be as in Example~\ref{example: lottery}. Assume $c=o(n)$ and $c$ is bounded away from $0$. Then:
	\begin{enumerate}[(i)]
		\item $T_{\frac{1}{2}} =(1+o(1))\frac{n\log 2}{c}$;
		\item $\mathbb{E}T=(1+o(1))\frac{n}{c}$;
		\item $\lim_{n\rightarrow \infty}\mathbb{P}\left(T> \frac{xn}{c}\right) =e^{-x}$ for any fixed $x\geq 0$. 
	\end{enumerate}
\end{proposition}
\begin{proof}
By Theorem~\ref{theorem: erdos renyi}, whp the $1$-uniform exchangeable coupon collector does not cover $[n]$ in  time less than $\frac{1}{2}n \log n$. Thus for time $t<\frac{1}{2}n\log n$, whp $T\leq t$ if and only if we have `won the lottery' by time $T$, that is, if $X_i=[n]$ for some $i\leq t$. This event occurs with probability $1-(1-p)^t$.

To obtain part (i) of the proposition, we observe that if $t\geq T_{\frac{1}{2}}$ then
\[\frac{1}{2}+o(1) \leq 1-(1-p)^t,\]
yielding $t\geq (1+o(1))\frac{\log 2}{\log (1-p)}=(1+o(1))\frac{n\log 2}{c}$, and we show similarly that if $t\leq T_{\frac{1}{2}}$ then $t\leq (1+o(1))\frac{n \log 2}{c}$ to conclude.

For part (ii), let $T'$ be the time at which we first `win the lottery' by receiving all of $V$ as our coupon.  We have 
\[\mathbb{E}T'=\sum_t tp(1-p)^{(t-1)}= \frac{1}{p}=\frac{n}{c}.\]
Since $T\leq T'$, we have that $\mathbb{E}T\leq \mathbb{E}T'$. Now from our estimates for the probability of winning the lottery by time $t$ above, whp we have $T'=o(n\log n)$. Thus by Theorem~\ref{theorem: erdos renyi}, whp $T = T'$, and
\begin{align*}\mathbb{E}T&\geq \sum_t t\mathbb{P}(T'=t\vert T'=T) \mathbb{P}(T'=T)\geq \sum_t t \left(\mathbb{P}(T'=t)-o(1)\right) =(1+o(1))\mathbb{E}(T'),
\end{align*}
whence we are done.

Finally for part (iii), we simply note that whp $T=T'$, and that
\begin{align*} 
\mathbb{P}(T'> \frac{xn}{c})&=(1-p)^{\frac{xn}{c}}=e^{-x(1+O(n^{-1}))}\rightarrow e^{-x}.
\end{align*}
\end{proof}

Proposition~\ref{prop: lottery covering} shows two things. First of all, the lower bound on $T_{\frac{1}{2}}$ in Proposition~\ref{proposition: elementary bounds} is essentially sharp; indeed taking $c=c(n)$ tending 
	to infinity slowly, we have in Example~\ref{example: lottery} that $\mu=\mathbb{E}\vert X\vert= c(1+o(1))$, and $T_{\frac{1}{2}}=(1+o(1))\frac{n\log 2}{\mu}$. 
	Further, by varying the value of $c=c(n)$ from $\Omega(\frac{1}{\log n})$ to $o(n)$, we can get $T_{\frac{1}{2}}$ to take asympotically any value between the bounds 
	from Proposition~\ref{proposition: elementary bounds}.

Secondly, we cannot in general expect $T$ to be sharply concentrated: part (iii) of Proposition~\ref{prop: lottery covering}  shows that we do not get sharper 
	concentration than the  Bollob\'as--Thomason-type concentration guaranteed by Proposition~\ref{cor: Bollobas-Thomason concentration}.  With this in mind, we next focus on conditions on $X$ which guarantees sharp concentration of the covering time $T$ and/or `slow coverage'.

%-----------------------------------------------------------------------------------------------------------------------------------------------------------------------------------------------------------------------------
\section{General concentration results}\label{section: slow coverage}
Let $V=[n]$ and $X$ be a random coupon variable for $V$. In this section we prove general results establishing (simple, easily checkable) sufficient conditions for sharp concentration of the covering time $T(\mathbf{X})$. We also include some results in the special case where the random coupon variable $X$ is balanced, transitive or exchangeable. Before stating our results, we need to introduce some notation.

Our proof strategy involves approximating the discrete-time process of collecting coupons by a continuous-time process. Instead of the coupon collector drawing a random coupon  $X$ at integer time points, she draws a random coupon $X$ (from the same distribution) at times given by a Poisson process with parameter $1$. 

 The times at which any given coupon is drawn will then be a thinned Poisson process, and the Poisson processes associated with different coupons will be independent. The times at which any particular element $x\in V$ is drawn will also be a thinned Poisson process, though the Poisson processes associated with different elements $x,y \in V$ will not in general be independent. Working in the continuous rather than in the discrete setting will greatly simplify calculations.

For $S\subseteq [n]$, set $h(S) = \Pr(X=S)$. For every $S$ with $h(S)>0$, start a Poisson process $\mathcal{P}_S$ with intensity $h(S)$. Each time an event occurs in $\mathcal{P}_S$, the coupon collector draws the coupon  $S$.
Let ${q_x := \sum_{S\ni x} h(S)}$ be the total intensity of all coupons covering $x$, and let ${q_{xy}:=\sum_{S\ni x,y} h(S)}$ be the total intensity of all coupons covering $x$ and $y$ simultaneously. Equivalently, $q_x = \Pr(x\in X)$ and $q_{xy} = \Pr(x,y\in X)$. Note that $\sum_{x\in V} q_x =\sum_{S}|S|\cdot h(S)= \E X=:\mu$.

Let $Z_{x,t}$ be the indicator event of the element $x$ not being covered at time $t$, and let $Z_t := \sum_{x\in V} Z_{x,t}$. An element $x$ has not been covered by time $t$ if its associated Poisson process with intensity $q_x$ has had no events in the time interval $[0,t]$. The probability of this occurring is $e^{-q_x t}$, and so the first two moments of $Z_t$ are:
\[
\E Z_t = \sum_{x\in V} e^{-q_x t}
\quad\textrm{ and }\quad
\E Z_t^2 = \sum_{x,y \in V} e^{-(q_x+q_y-q_{xy}) t}.
\]
Many of our proofs will use the second moment method to show concentration of $Z_t$, which implies concentration of $T(\bold X)$.
%-------------------------------------------------------------------------------------------------------
\subsection{Results}\label{subsection: results}
Let $V=V(n)$ be an $n$-set and $X=X(n)$ a covering random variables for $V$ (formally we consider sequences $(V(n))_{n\in \mathbb{N}}$ and $(X(n))_{n\in \mathbb{N}}$). We let $T=T(\mathbf{X})$ denote the covering time for the $X$-coupon collector.

Our first result gives us sufficient conditions for $\mathbb{E} Z_t$ (the number of uncollected elements) to have a sharp transition from $\omega(1)$ to $o(1)$. This holds trivially for balanced coupon collectors, but may fail if $X$ is far from balanced --- for instance, if some elements of $V$ occur very rarely. 

For any $\alpha\in \R$, let $\|\mathbf{q}\|_\alpha$ be the $\alpha$-H\"older mean of the vector of intensities $\mathbf{q}=(q_x)$,  i.e. $\|\mathbf{q}\|_\alpha := \left(\frac{1}{n}\sum_x q_x^\alpha\right)^{\frac{1}{\alpha}}$ (with the usual convention that for $\alpha=0$, $\|\mathbf{q}\|_0$ is the geometric mean $(\prod_x q_x)^{1/n}$). 
\begin{theorem}\label{main thm: sharp first moment} Set $q_\star := \min_x q_x$. If any of the following conditions is satisfied, then there exists $T^-(n)$ and $T^+(n)$ such that $T^- = (1+o(1))T^+$, $\E Z_{T^-} \to \infty$ and $\,\E Z_{T^+} \to 0$. 
	\begin{enumerate}[(i)]
		\item There exists $t \gg q_\star^{-1}$ such that $\E Z_{t} \gg 1$
		\item There exists $\alpha=o(\log n)$ such that $(\alpha+2) \|\mathbf{q}\|_{-\alpha} \leq \log n  \cdot q_\star  $
		\item There exists  $A_r = \exp(\omega(r))$, which does not depend on $n$, such that for any $r>0$ and all sufficiently large $n$, the number of $y\in V$
		satisfying $q_y < r q_\star$ is at least $A_r$.
	\end{enumerate}
\end{theorem}
If $\frac{q_x}{q_y} \leq \frac{1}{2}\log n$ for all $x,y$, condition (ii) is met trivially with $\alpha=0$; in fact it can be shown that the factor $\frac{1}{2}$ can be replaced by any positive number. So in particular Theorem~\ref{main thm: sharp first moment} applies to `almost balanced' random coupon variables $X$.

Our second result gives whp bounds on $T$ when correlations are bounded.
\begin{theorem}	\label{main thm: bounded correlation}
	If there exist $C=C(n)$, $T^-=T^-(n)$ and $T^+=T^+(n)$ such that all of the following are satisfied:

	\begin{enumerate}[(i)]
		\item  $\E Z_{T^-}\to \infty$ and $\E Z_{T^+}\to 0$,
		\item the coupons have $C$-bounded correlation, i.e. $q_{xy} \leq C q_x q_y$ for all $x \neq y$,
		\item $C \avq  = o(\frac{1}{\log n})$, where ${\avq = \E\big[q_\theta \big| \theta \textrm{ not covered at time } T^-\big]}$ is the expected size of $q_{\theta}$ for $\theta$ drawn uniformly at random from among the uncovered vertices at time $T^-$,
	\end{enumerate}
	then $T^- \leq T(\mathbf{X}) \leq  T^+$ whp.
\end{theorem}
The parameter $\avq$ should be viewed as the `speed' of covering at time $T^-$, 
and it is usually hard to compute exactly. However in order to apply Theorem~\ref{main thm: bounded correlation} it is enough to give an upper bound on $\avq$. The simplest such bound, namely $\avq \leq \max_x q_x$, can easily be improved. For instance, it is straight-forward to show that $\avq \leq \|\mathbf{q}\|_{-\alpha}$ for any finite $\alpha$ and all $n$ sufficiently large. This makes condition (iii) easy to check in many situations.

We can obtain further results when $X$ is assumed to be balanced. The next theorem tells us that if either the coupons are `small' (size $o(n)$) or the pairwise correlations between the elements of $V$ are `not too strong' then we have sharp concentration for $T$.
\begin{theorem}\label{main thm: balanced coupons}
Let $X$ be a balanced random coupon variable with $\mu := \E |X|$.
	\begin{enumerate}[(i)]
		\item If there exists $t$ such that $\sum_{x,y}( e^{q_{xy} t} -1 ) = o(n^2)$ and $\sum_{x}e^{-q_{x} t} =\omega(1) $, then $T(\mathbf X)\geq t$ whp.
		\item If there exists $1< \beta(n) < n$  tending to infinity and $q=o\left(\frac{\mu}{n\log \beta}\right)$ with $q_{xy}\leq q$ for all but at most $\frac{1}{\beta} n^2$ `bad' pairs $(x,y)$,
		then $T(\mathbf X)\geq \frac{n }{\mu}(\log\beta-\omega(1))$ whp.
		
		In particular, if $q = o\left(\frac{\mu}{n\log n}\right)$ and there are at most $n^{1+o(1)}$ such `bad' pairs, then \\ ${T(\mathbf X) = (1\pm o(1))\frac{n \log n}{\mu}}$ whp.
		\item If all coupons have size at most $M$, then ${T(\mathbf X)\geq \frac{n}{\mu}(\log n - \log M - \omega(1))}$ whp, for $\omega(1)$ tending to infinity arbitrarily slowly.
		
		In particular, if $M = n^{o(1)}$, then ${T(\mathbf X) = (1+ o(1))\frac{n \log n}{\mu}}$ whp.
		\item If $q_{xy} = q_{x'y'}$ for all $x\neq y,x'\neq y'$, and all coupons have size at most $M$, and $T^-$ is such that ${T^- = \frac{n}{\mu} \cdot \min\big(o(\frac{n}{M}),\log n - \omega(1)\big)}$, then $T^- \leq T(\mathbf X)$ whp.
		
		In particular, if $M = o(\frac{n}{\log n})$, then ${T(\mathbf X) = (1+ o(1))\frac{n \log n}{\mu}}$ whp.
	\end{enumerate}
\end{theorem}

There are examples where the whp lower bounds given by this theorem are sharp, while the upper bounds given by the first moment method are not; see for instance Example~\ref{example: collecting a smaller set} in Section~\ref{section: fast coverage}. Note also that unlike Theorem~\ref{theorem: aldous}, Theorem~\ref{main thm: balanced coupons} also locates the threshold.

For balanced  random covering variables  we also have good control for both concentration and the covering time when $X$ satisfies an `almost negative correlation' condition. Here below we say that a function $m=m(n)$ is \emph{sub-polynomial} in $n$ if $m = n^{o(1)}$.
\begin{theorem}\label{main thm:  balanced k-uniform}
	Let $\delta>0$ be fixed. Let $X$ be a balanced covering random variable for an $n$-set $V$, with $\Pr(x \in X) = c$ for some 
	$c\in(0,1-\delta)$. Suppose  further that we have \emph{almost negative correlations}, namely that there exist $\eta=o(1/\log n)$ 
	and $b=b(n)$ subpolynomial in $n$ such that for any 
	$x\in V$ \[\Pr(x,y \notin X) \leq (1-c)^2 (1+\eta).\] 
	holds for all but at most $b$ elements $y$. Then 
	whp $T(\mathbf{X})=(1+o(1))\frac{\log n}{-\log(1-c)}$.
\end{theorem}
Note that if $\eta=0$ then the correlation condition is the same as the commonly used pairwise negative correlation condition. Recently a substantial theory for negatively correlated random variables has been developed and numerous common examples have been shown to have this and even stronger correlation properties, see~\cite{BBL}.
\begin{cor}\label{main thm: exchangeable uniform} 
	Suppose that $X$ is balanced, has pairwise negative correlation, and $\Pr(x \in X) = c \leq 1- \delta$ for a fixed $\delta>0$. Then whp $T(\mathbf{X})=(1+o(1))\frac{\log n}{-\log(1-c)}$. \qed	
\end{cor}

If $c=o(1)$ then the equality above may be rewritten as $T(\mathbf{X})=(1+o(1))\frac{n\log n}{\mathbb{E}\vert X\vert}$.

We next give conditions implying sharp concentration for the covering time of an  
exchangeable random variable $X$ 
 around the same value as a \emph{uniform} exchangeable random variable with the same mean coupon size.
\begin{theorem}\label{main thm: exchangeable coupons}
	Let $X$ be an exchangeable random coupon variable, for $V=[n]$, with maximum coupon size $M$, average coupon size $\mu$ and mean square coupon size $\chi$. If any of the four conditions below holds, then whp $T(\mathbf{X})=(1+o(1)) \frac{\log n}{-\log (1-\frac{\mu}{n})}$ (which in the case $\mu=o(n)$ can be rewritten as $T=(1+o(1))\frac{n\log n}{\mu}$).
	\begin{enumerate}[(i)]
		\item $M=o(\sqrt{n\log n})$; %i.e. small coupons
		\item $M=o(n)$ and $M=o(\sqrt{\mu n\log n})$; %i.e. large coupons but not much larger than \mu
		\item $M=o(n)$ and $\chi=o(\mu n \log n)$;  %i.e large coupons but not too much variation
		\item $\mu<(1-\delta)n$ for some $\delta>0$ and  $\chi =(1+o(\frac{1}{\mu n\log n}))\mu^2$.	%concentration of the coupon size around \mu
	\end{enumerate}
\end{theorem}
Note that the theorem includes the case when $X$ is $k$-uniform for $k=cn$.  Roughly speaking, the conditions in the theorem move from small coupons, with no other assumptions, to larger coupons where successively stronger  size concentration is needed.

In some applications  it is useful to have more accurate information about the sharpness of the concentration. We thus include a final result on the cover time $T$ for the $k$-uniform exchangeable coupon collector in the sublinear case $k=o(n)$, in the spirit of the theorem of Erd{\H o}s and R\'enyi (Theorem~\ref{theorem: erdos renyi}) mentioned in the introduction. 
\begin{theorem}\label{thm: k-uniform exchangeable coupon collector, sublinear k}
	If $k=o\left(n\right)$, then the covering time $T$ for a $k$-uniform exchangeable coupon collector is sharply concentrated around $\frac{n \log n}{k}$. More precisely, 
	we have $\Pr\Big(\vert T - \frac{n \log n}{k}\vert > \frac{cn}{k}\Big) \to e^{-c}$ as $n\to \infty$. 
\end{theorem}

\subsection{Continuous-time approximation of the coupon collector}
In this subsection, we formalize our approximation of the discrete-time coupon collector by a continuous-time process. As described above, for every subset $S\subseteq V$ with $h(S)=\Pr(X=S)>0$, we start at time $t=0$ a Poisson process $\mathcal{P}_S$ with intensity $h(S)$. Our continuous coupon collector receives $S$ as a coupon each time an event occurs in $\mathcal{P}_S$. List the coupons in the order they are received by the continuous collector as $S_1, S_2, S_3, \ldots $. The distribution of the sequence $\mathbf{S}=(S_n)_{n\in \mathbb{N}}$ is identical to that of the sequence of coupons $\mathbf{X}$ received by the (discrete-time) $X$-coupon collector. Furthermore, the time $t_m$ at which the continuous coupon collector receives his $m^{\textrm{th}}$ coupon is sharply concentrated around $m$. Indeed, by a standard bound on the Poisson distribution, for any $\varepsilon>0$,
\[\Pr\left(\vert t_m-m\vert\geq \varepsilon m \right)=O\left(\frac{1}{\sqrt{m\varepsilon^2}}e^{-\frac{m\varepsilon^2}{2}} \right).\]
In particular, provided the covering time for the continuous coupon collector $t_T$ is large (grows with $n$), we have that whp $t_T=(1+o(1))T$. Thus it is enough to prove whp bounds on $t_T$ to establish whp bounds on $T$. We shall thus in a slight abuse of notation identify $t_T$ with $T$ in the rest of the paper, and prove bounds for the covering time via the continuous coupon collector. In particular we shall set $T=\inf\{t: \ Z_t=0\}$.

%-------------------------------------------------------------------------------------------------------
\subsection{Proofs: concentration of the covering time}\label{subsection: proofs: concentration}
It will be useful to consider the function $f(t) = \log( \E Z_t )$. The first two derivatives of $f$ are
\[f'(t) = - \frac{\sum_{x} q_x e^{-q_x t}}{\sum_{x} e^{-q_x t}}\leq 0, \quad f''(t) =\frac{1}{2} \cdot \frac{\sum_{x,y} (q_x-q_y)^2 e^{-(q_x+q_y) t}}{\sum_{x,y} e^{-(q_x+q_y) t}}\geq 0,\]
from which we can see that $f$ is a decreasing convex function. In particular for any $t\geq 0$,
\begin{equation}\label{eq: f'(t) bound}
f(t)-tf'(t) \leq f(0)=\log n.
\end{equation}
Similarly, for any $t>0$ and $\Delta<t$,
\begin{equation}\label{eq: f(t pm Delta) bound}
f(t-\Delta)-f(t)\geq -\Delta f'(t) \qquad \textrm{ and } \qquad f(t)-f(t+\Delta)\geq -\Delta f'(t).
\end{equation}
Finally, note that ${f'(t) = -\E[q_\theta | \theta \textrm{ not covered at time } t]}$. The following lemma gives the basic first and second moment bounds on the covering time.
\begin{lemma}	\label{basiclemma}
	Let $T=T(\mathbf{X})$ be the covering time for a coupon collector $\mathbf{X}$, and let $(q_x)$ and $(q_{xy})$ be its associated single and pairwise intensities.
	\begin{enumerate}
		\item If $t=t(n)$ is such that $\sum_{x} e^{-q_x t}\to 0$,  as $n\to \infty$,
		then $T  \leq t$ whp.
		
		\item If $t=t(n)$ is such that $\frac{\sum_{x\neq y} (e^{q_{xy} t}-1)\cdot e^{-(q_x+q_y) t}}{\sum_{x,y} e^{-(q_x+q_y) t}} = o(1)$ and $\sum_x e^{-q_x t} \to \infty$, then $t \leq T$ whp.
	\end{enumerate}
\end{lemma}
\begin{proof}We divide the proof into two parts.

\noindent \textbf{Part 1.} 
%This is the first moment bound on $T$: 
Suppose $t=t(n)$ satisfies the lemma's assumption. By Markov's inequality		
		$\Pr(t \geq T) = \Pr(Z_t >0) \leq \E Z_t = \sum_x e^{-q_x t} \to 0$, so $t<T$ whp.

\noindent \textbf{Part 2.} %This is a second moment bound on $T$: 
Suppose $t=t(n)$ satisfies our assumption. Then $\E Z_t \to \infty$, and
		\begin{align*}
		\frac{\Var[Z_t]}{\E[Z_t]^2}
		&=\frac{\sum_{x,y} (e^{q_{xy} t}-1)\cdot e^{-(q_x+q_y) t}}{\sum_{x,y} e^{-(q_x+q_y) t}}
		\\
		&< \frac{\sum_{x\neq y} (e^{q_{xy} t}-1)\cdot e^{-(q_x+q_y) t}}{\sum_{x,y} e^{-(q_x+q_y) t}} + \frac{1}{\sum_{x} e^{-q_x t}} 
		= o(1) + o(1),
		\end{align*}
		so by Chebyshev's inequality $Z_t = (1+o(1))\E Z_t \to \infty$ whp, so that whp $Z_t>0$ and $T>t$.
		%${\Pr(t \geq T)} = {\Pr(Z_t >0)} = {1-o(1)}$, so $t \geq T$ whp.
	\end{proof}

\begin{proof}[Proof of Theorem~\ref{main thm: bounded correlation}]
	Assumption (i) gives us $\E Z_{T^+} \to 0$, from which it is immediate by Lemma~\ref{basiclemma} part 1 that whp $T \leq  T^+$. To establish the lower bound on $T$, we 
	shall consider the set of `rare' coupons ${U  = \{x \in V: q_x \leq 2 \avq \}}$. Let $Y_t$ be the number of  $ x\in U$ for which $x$ is uncovered at time $t$. 
		\begin{claim}
			\label{Yt-inf}
			$\E Y_{t} \to \infty$ for any $t \leq T^-$
		\end{claim}
		\begin{proof}
		We bound $\avq$ from below to get:
		\[
		\avq
		\geq \frac{\sum_{x\in V \backslash U} q_x e^{-q_x {T^-}}}{\sum_{x \in V} e^{-q_x {T^-}}}
		\geq 2\avq \frac{\sum_{x\in V \backslash U} e^{-q_x {T^-}}}{\sum_{x\in V} e^{-q_x {T^-}}}
		= 2\avq \cdot  \left(1-\frac{\E Y_{T^-}}{\E Z_{T^-}}\right).
		\]
		Dividing both sides by $\avq$ gives us $1 \geq 2\big(1-\frac{\E Y_{T^-}}{\E Z_{T^-}}\big)$, which implies $\E Y_{T^-} \geq \frac{1}{2}\E Z_{T^-}$. Since by assumption (i) $\E Z_{T^-} \to \infty$, and since $\E Y_t$ is decreasing in $t$, we must have that $\E Y_{t} \to \infty$ for any $t \leq T^-$, as claimed.
	\end{proof}

	Now, as observed after inequality (\ref{eq: f(t pm Delta) bound}), $f'(t) = -\E[q_\theta | \theta \textrm{ not covered at time } t]$, and in particular $f'(T^-) =  -\avq$. By assumption (i) $f(T^-)\to \infty$, so inequality (\ref{eq: f'(t) bound}) gives 
	\begin{equation}\label{eq: bound on Tq}T^-\cdot \avq \leq f(T^-)-T^-f'(T^-)\leq \log n  \end{equation}
We are now in a position to apply part 2 of  Lemma~\ref{basiclemma} to the \emph{restriction} of the coupon collector to the set of rare coupons $U$ (i.e. the coupon collector with covering variable $X\cap U$). %For this we need an upper bound on $q_{xy} t$ for $t \leq T^-$ . 
For any $x\neq y$, we have that ${q_{xy}t \leq C q_x q_y t}$ by assumption (ii). If $x,y\in U$ this quantity is at most ${4C  (\avq)^2 T^-}$. By inequality (\ref{eq: bound on Tq}) and our assumption (iii), ${4C  (\avq)^2 T^- \leq 4C \avq \log n = o(1)}$. Thus
	\begin{align*}
	\label{neg-cor-var}
	\frac{\sum_{x,y \in U: x \neq y } (e^{q_{xy} {t}}-1)\cdot e^{-(q_x+q_y) {t}}}{\sum_{x,y \in U} e^{-(q_x+q_y) {t}}}
	\leq (e^{4C \avq \log n}-1) \cdot  \frac{ \sum_{x,y \in U: x \neq y} e^{-(q_x+q_y) {t}}}{\sum_{x,y \in U} e^{-(q_x+q_y) {t}}} 
	= o(1).
	\end{align*}
	Since by Claim~\ref{Yt-inf} $\E Y_t \to \infty$, we have by Lemma~\cref{basiclemma} part 2 that whp $T^- \leq \inf \{ t: Y_t = 0\}$ whp. Since by construction $Y_t\leq Z_t$, this gives $T^-\leq T$ whp, as required.
\end{proof}
% % % % % % % % % % % % % % % % % % % % % % % % % % % % % % % % % % % % % % % % % % % % % % % %
\subsection{Proofs: sharp transition for $\mathbb{E}Z_t$}
\begin{proof}[Proof of Theorem~\ref{main thm: sharp first moment}]
	Let $T^* = T^*(n) $ be the unique real for which $\E Z_{T^*} = 1$.

\noindent \textbf{Part (i).}
			By our assumption, we can find $\Delta=\Delta(n)$ such that  ${T^*\gg \Delta  \gg \frac{1}{\min_x q_x}}$. We will show that ${T^- := T^*-\Delta}$ and  ${T^+ := T^*+\Delta}$ have the desired properties. By definition of $\Delta$, we have $T^-=(1+o(1))T^+$.
			Now $-\Delta f'(T^*)\geq \Delta \min_x q_x\gg 1$, so by inequality~(\ref{eq: f(t pm Delta) bound}) we have $f(T^*-\Delta)-f(T^*)$ and  $f(T^*)-f(T^*+\Delta)$ both tending to infinity. Since $f(T^*) = 0$, this implies that $\E Z_{T^*-\Delta} \to \infty$ and $\E Z_{T^*+\Delta} \to 0$, as required.

\noindent \textbf{Part (ii).} Let $\alpha(n)=o(\log n)$ be as in the assumption.
			Pick $1\ll c \leq \log n/(\alpha+2)$, and set $t^* = \frac{\log n - \alpha}{\|\mathbf{q}\|_{-\alpha}}$. By assumption, $t^*\gg \min_x q_x$.

			For any $x\in V$, we have 
			\[
			{q_x t^* \geq \frac{ (\alpha+2)\|\mathbf{q}\|_{-\alpha}t^*}{\log n}} = \alpha +2-\frac{(\alpha+2)c}{\log n} \geq \alpha+1.
			\]
			Now the function $z \mapsto e^{-z^{-1/\alpha}}$ is convex over those $z$ satisfying $z^{-1/\alpha} \geq \alpha+1$. We can 
			therefore apply Jensen's inequality as follows:
	\[		\E Z_{t^*}=\sum_{x\in V} e^{-q_x t^*} =\sum_{x\in V} e^{-\left(q_x t_*\right)^{-\alpha \cdot (-1/\alpha)}} \geq n \exp(-\|\mathbf{q}\|_{-\alpha} t^*) = e^c \to \infty.
			\]
			This gives us a $t^*\gg \min_x q_x$ such that $\E Z_{t^*} \gg 1$. We are then done by part (i).

\noindent \textbf{Part (iii).}
			Let the function $A_r$ be as in the assumption. Let $R = R(n)$ be the largest $r$ such that there are at least  $A_r$ elements $y$ 
			with $q_y \leq r\min_x q_x$; $R$ is finite for every $n$, but by assumption tends to infinity as $n\rightarrow \infty$. We can 
			therefore find $t^*=t^*(n)$ satisfying 
			\[\frac{1}{\min_x q_x} \ll t^* \ll \frac{\log A_{R}}{R\min_x q_x }.\]
			We now bound $\E Z_{t^*}$ from below:
			\[\E Z_{t^*} \geq \!\!\!\!\sum_{\substack{y \in V: \\  q_y \leq R \min_x q_x}} \!\!\!\!e^{-q_y t^*} \geq  A_{R} e^{-R \min_x q_x t^*}\gg 1,\]
			by the choice of $t^*$. We are then done by part $(i)$ .
\end{proof}

%-------------------------------------------------------------------------------------------------------
\subsection{Proofs: balanced coupons}

\begin{proof}[Proof of Theorem \ref{main thm: balanced coupons}]
	Since $X$ is balanced, $q_x =\mu/n$ for all $x\in V$. We will show that we can apply part \textit{2} of Lemma~\ref{basiclemma} provided \textit{(i)} holds, and then that each of conditions \textit{(ii)--(iv)} implies \textit{(i)}. The `in particular' statements in (ii)--(iv) combine the lower bound given by those special cases with the upper bound on $T$ from Proposition~\ref{proposition: elementary bounds}.

		\noindent \textbf{Condition (i).} 
		Since $q_x = q_y$ for all $x,y$, we have
		\[
		\frac{\sum_{x,y} (e^{q_{xy} t}-1)\cdot e^{-(q_x+q_y) t}}{\sum_{x,y} e^{-(q_x+q_y) t}}  = \frac{\sum_{x,y} (e^{q_{xy} t}-1) }{n^2} = \frac{o(n^2)}{n^2}=o(1).
		\]
		We can therefore apply part \textit{2.} of Lemma~\ref{basiclemma} to conclude that $T(\mathbf X)\geq T^-$.

		\noindent \textbf{Condition (ii).} 
Set $t=\frac{n (\log \beta -\omega(1))}{\mu}$ for some $\omega(1)$ tending to infinity arbitrarily slowly. Note 
\[Z_t= ne^{-\log \beta +\omega(1)}\geq e^{\omega(1)} \rightarrow +\infty.\] Let $E$ be the set of exceptional pairs $(x,y)$ with $q_{xy}>q$. Since $q_{xy} \leq q$ for $(x,y)\notin E$ and $q_{xy} \leq q_x=\frac{\mu}{n}$ for $(x,y)\in E$, 
we have:
		\begin{align*}
		\sum_{(x,y)\notin E} e^{q_{xy} t} &\leq n^2 e^{qt} \leq n^2 + o(n^2), \qquad \textrm{and}\\
		\sum_{(x,y)\in E} e^{q_{xy} t}    & \leq         \frac{n^2}{\alpha} \cdot e^{\frac{\mu t}{n}}      =        \frac{n^2}{\alpha}    \cdot e^{\log a - \omega(1)}    = o(n^2),
	\end{align*}
		Together, these bound give that $\sum_{x,y} e^{q_{xy} t} = n^2+o(n^2)$. Hence condition \textit{(i)} is satisfied for our choice of $t$, and we are done.

		\noindent \textbf{Condition (iii).}
	Fix $x\in V$, and consider the sum $\sum_{y\in V} q_{xy}$. Each subset $X\subseteq V$ containing $x$ contributes $h(X)$ to $\vert X\vert$ terms of the sum. Thus
		\[
		\sum_{y\in V} q_{xy} = \sum_{X: \ x\in X} \vert X\vert h(X) \leq M \sum_{X: \ x \in X} h(X)=Mq_x.
		\]
		Furthermore, for every $y$, $q_{xy}\leq q_x=\frac{\mu}{n}$. We ask therefore: which choices of $\tilde q_{xy}$, subject to the constraints $\sum_{y\in V} \tilde q_{xy}\leq M\frac{\mu}{n}$ and $0\leq \tilde q_{xy}\leq \frac{\mu}{n}$, maximize the expression $\sum_{y\in V} e^{\tilde q_{xy}t}-1$?
		Since $z\mapsto e^{zt}-1$ is an increasing function for $t>0$, the optimal $\tilde q_{xy}$ must satisfy ${\sum_{y\in V} \tilde q_{xy}=M\frac{\mu}{n}}$. 
		By the Karamata inequality 
		the maximum of the sum is then attained when $M$ of the $\tilde{q_xy}$ are equal to $\frac{\mu}{n}$ and the rest are equal to $0$. Thus
		\[
		\sum_{y\in V} (e^{q_{xy}t}-1) \leq \sum_{y\in V} (e^{\tilde q_{xy}t}-1) < M \cdot e^{\frac{\mu t}{n}}.
		\]
		Setting $t = \frac{n(\log n - \log M - \omega(1))}{\mu}$ for an arbitrary $\omega(1)$ tending to infinity, and summing over all $x$, we get
		\[
		\sum_{x,y\in V} (e^{q_{xy} t}-1) < Mn\cdot e^{\frac{\mu t}{n}} = Mn \cdot e^{\log n - \log M - \omega(1)} = o(n^2).
		\]
		Since in addition our choice of $t$ ensures $\E Z_t=M e^{\omega(1)}\rightarrow +\infty$, condition \textit{(i)} is satisfied, and we are done.

		\noindent\textbf{Condition (iv).} If $q_{xy}=q$ for all $x\neq y$ and some $q$, then
		\[
		M\mu \geq \sum_{x,y} q_{xy} = \mu+\sum_{x\neq y} q_{xy} = \mu+ n(n-1) q,
		\]
		so $q \leq \frac{(M-1)\mu}{(n-1)n}$. For $t \leq \frac{n(\log n - \omega (1))}{\mu}$ with $t = o\big(\frac{n^2}{M \mu}\big)$, we have that $qt =o(1)$ and ${q_x t \leq  \log n - \omega (1)}$, whence $Z_t\rightarrow \infty$ and
		\begin{align*}
		\sum_{(x,y): \ x\neq y} (e^{q_{xy}t}-1) +\sum_{x} (e^{q_{x}t}-1)
		 < n^2 (e^{o(1)}-1)+ ne^{\log n - \omega (1)} 		
		=o(n^2).
		\end{align*}
		Hence condition \textit{(i)} is satisfied once more, and we are done.

\end{proof}

\begin{proof}[Proof of Theorem~\ref{main thm:  balanced k-uniform}]
Since $X$ is balanced, we have that $t_0=\frac{\log n}{-\log(1-c)}$ is a first-moment threshold for the expected number of uncovered vertices $\mathbb{E}\vert V\setminus C_t\vert= n(1-c)^t$. In particular we have that for any fixed $\varepsilon>0$ the covering time $T=T(\mathbf{X})$ satisfies $T<(1+\varepsilon)\frac{\log n}{-\log(1-c)}$. We turn our attention to the variance of $\vert V\setminus C_t\vert $ to show concentration of its value just below the first-moment threshold $t_0$.
\begin{align*}
&\E \left[\vert V\setminus C_t\vert^2\right]=\sum_{x,y} \Pr (x,y \notin C_t)
=\sum_{x,y} (1-\Pr(x,y \in X))^t \\
&\leq n\bigl((1-c)^{2t}(1+\eta)^t(n-b)+ (1-c)^tb\bigr)
< n^2(1-c)^{2t}\left( (1+\eta)^t + b\left(\frac{1}{n(1-c)^{t}}\right)   \right).
\end{align*}
Now for $\varepsilon>0$ fixed and $t\leq (1-\varepsilon)\frac{\log n}{-\log(1-c)}$, our assumptions on $b$ and $\eta$ tell us that the above is at most
\begin{align*}
&\ \left(\E \bigl[\vert V\setminus C_t\vert\bigr] \right)^2\left( e^{\frac{\eta \log n}{-\log (1-c)}} + \frac{b}{n^{\varepsilon}}   \right)\\
&=\left(\E \bigl[\vert V\setminus C_t\vert\bigr] \right)^2 (1+o(1)).
\end{align*}
Chebyshev's inequality is then enough to give us concentration of $\vert V\setminus C_t\vert$ about its (large, non-zero) mean for these values of $t$. In particular whp $T>(1-\varepsilon)\frac{\log n}{-\log(1-c)}$. Thus whp $T=(1+o(1))\frac{\log n}{-\log(1-c)}$, as required.
\end{proof}

%-------------------------------------------------------------------------------------------------------
\subsection{Proofs: exchangeable coupons}\label{subsection: proofs: exchangeable coupons}
In the case where $X$ is an exchangeable random variable, we exhibit a (natural) coupling between the process of covering $V$ by $X$ with the classical coupon collector problem (covering by singletons chosen uniformly at random), which allows us to determine (up to a small error) the expectation of the covering time $T$ as well as, in the case where $\vert X\vert =o(n)$ holds whp, to prove that $T$ is concentrated around its mean. We note that a similar coupling appears in a work of Sellke~\cite{Sellke95}, though it is used for a different purpose.

We begin by proving Theorem~\ref{thm: k-uniform exchangeable coupon collector, sublinear k}. Let $k=k(n)$ be a sequence of natural numbers. Set $V=V(n)=[n]$, and let $X=X(n)$ be the random covering variable for $V$ obtained by selecting a $k$-set from $V$ uniformly at random. Let also $Y=Y(n)$ be the classical random coupon variable for $V$, namely the random covering variable obtained by selecting a singleton from $V$ uniformly at random. 

\begin{proof}[Proof of theorem \ref{thm: k-uniform exchangeable coupon collector, sublinear k}] We couple the $k$-uniform coupon sequence $\mathbf{X}$ to the sequence of coupons received by the $Y$-coupon collector, $\mathbf{Y} = (Y_i)_{i=1}^\infty$. For natural numbers $a\leq b$, set $C_Y[a,b]:=\bigcup_{i \in [a,b]} Y_i$. 
	Let $a_0 = 0$, and define $a_i$, $i\geq 1$, recursively to be the least integer such that ${\left\vert C_Y[a_{i-1}\!+\!1,a_i]\right\vert = k}$. Next, let  ${X_i = C_Y[a_{i-1}\!+\!1,a_i]}$. Clearly, the $X_i$ obtained are independent random sets, uniformly distributed among the $k$-sets in $V$, so $(X_i)_{i=1}^\infty \sim \mathbf{X}$. Furthermore, the integers $\ell_i := a_i-a_{i-1}$ are i.i.d. random variables.

	This coupling between the coupon collectors enables us to relate $T(\mathbf{X})$ to $T(\mathbf{Y})$. For any natural number $t$, we have that
\[
\bigcup_{j=1}^{t} X_j= \bigcup_{i=1}^{a_t}Y_i,
\]
so $T(\mathbf{X})\leq t$ if and only if $T(\mathbf{Y})\leq {a_{t}}$. Conversely, $T(\mathbf{X})> t$ if and only if $T(\mathbf{Y}) > {a_{t}}$. In other words,
\begin{equation}
\sum_{i=1}^{T(\mathbf{X})-1} \ell_i  = a_{T(\mathbf{X})-1} < T(\mathbf{Y}) \leq a_{T(\mathbf{X})} = \sum_{i=1}^{T(\mathbf{X})} \ell_i.
\label{T1-TK}
\end{equation}
At this point, it is straightforward to get an estimate for $\mathbb{E}T(\mathbf{X})$ in terms of the (well--known) expectations of $T(\mathbf{Y})$ and $\ell_1$, via an application of Wald's inequality. To obtain sharp concentration for $T(\mathbf{X})$ we need only a little more work. Let $S_m:= \sum_{i=1}^m \ell_i$. We shall use the following lemma, establishing sharp concentration for $S_m$, together with the Erd{\H o}s--R\'enyi sharp concentration theorem for $T(\mathbf{Y})$ to deduce we have the desired sharp concentration for $T(\mathbf{X})$.
\begin{lemma}\label{Sm-concentration}
If $k=o(n)$, then for all $c>0$ and $m>\frac{n}{k}$ the following inequality holds:
\[
\Pr\left(\vert S_m-\E S_m\vert >c\cdot k\sqrt{\frac{m}{n}} \right)<4\cdot  e^{-c}.
\]
\end{lemma}
\begin{proof}[Proof of \cref{Sm-concentration}]
For $0\leq i \leq k-1$, let $\tau_i$ be the time it takes for the singleton collector to draw the $(i+1)^{\textrm{th}}$ distinct coupon after she has collected $i$ distinct coupons. Clearly,  $\tau_i \sim \Geom(\frac{n-i}{n})$. and has moment-generating function
\[
M_{\tau_i}(\lambda) :=\E[e^{\lambda \tau_i}] = \frac{(1-\frac{i}{n})e^\lambda}{1-\frac{i}{n}e^\lambda}.
\]
Note that $\ell_1 = \sum_{i=0}^{k-1} \tau_i$. Since $S_m = \sum_{i=1}^m \ell_i$ is the sum of $m$ independent copies of $\ell_1$, its moment generating function is given by
\[
M_{S_m}(\lambda)
= \left(\prod_{i=0}^{k-1} \frac{(1-\frac{i}{n})e^\lambda}{1-\frac{i}{n}e^\lambda}\right)^m.
\]
Applying Markov's inequality to the random variable $\exp{(\lambda {S_m})}$, for some $\lambda$: $\lambda\neq 0, \ \lambda =o(1)$ to be specified later, gives
\begin{align}
\Pr\left(e^{\lambda S_m} > e^{\lambda\E S_m+c}\right)
&<\frac{M_{S_m}(\lambda)}{\exp\left(\lambda\E S_m+c\right)} \nonumber
\\
&=\frac{\exp\left(m \sum_{i=0}^{k-1} (\lambda+\log(1-\frac i n)-\log(1-\frac i n e^\lambda))\right)}{\exp\left(m\left(\sum_{i=0}^{k-1} \frac{\lambda}{1-\frac i n}\right)+c\right)} \nonumber
\\
&=\exp\left(m \sum_{i=0}^{k-1} \Big(\lambda+\log\big(1-\frac i n\big)-\log\big(1-\frac i n e^\lambda\big)-\frac{\lambda}{1-\frac i n}\Big)\right)\nonumber
\\
&\leq\exp\left(m k \Big[\lambda+\log\big(1-\frac k n\big)-\log\big(1-\frac k n e^\lambda\big)-\frac{\lambda}{1-\frac k n}\Big]+c\right)\label{chernoffbound},
\end{align}
where the last inequality holds since the summands are on-decreasing in $i$ (this can be checked e.g. by computing the derivative of a summand with respect to $i$). We use a Taylor expansion of degree $d=\lceil-\log |\lambda|\rceil$ to estimate the quantity inside the square brackets.
\begin{equation}
\lambda+\log\bigl(1-\frac k n\bigr)-\log\big(1-\frac k n e^\lambda\big)-\frac{\lambda}{1-\frac k n} \leq \sum_{j=1}^{d} (e^{j\lambda}-j\lambda-1)\frac{k^j}{jn^j}+\frac{k^{d+1}}{(d+1)(n-k)^{d+1}}
\label{taylorexp}
\end{equation}
Note that $(e^{j\lambda}-j\lambda-1) =(\frac{1}{2}+o(1)) \cdot (j\lambda)^2$, since $j\lambda =o(1)$, whereas $\frac{k^{d+1}}{(n-k)^{d+1}} \ll (e^{-2})^d\cdot \frac{k}{n} \leq \frac{\lambda^2 k}{n}$, since $\frac{k}{n-k} \ll e^{-2}$ (since $k=o(n)$ by assumption). The right hand side of inequality (\ref{taylorexp}) can thus be bounded by 
\[
\left(\frac{1}{2}+o(1)\right)\cdot \lambda^2 \sum_{j=1}^d \frac{jk^j}{n^j} +  o\left(\frac{\lambda^2k}{n}\right)= \frac{(1+o(1))\lambda^2k}{2n}.
\]
Applying this bound to the right-hand side of inequality (\ref{chernoffbound}) gives us the following:
\begin{align}
\Pr\left(e^{\lambda S_m} > e^{\lambda\E S_m+c}\right)
&< \exp\left( \frac{(1+o(1))m\lambda^2k^2}{2n}+c     \right). \nonumber
\end{align}

Letting $\lambda = \pm\frac{1}{k}\sqrt{\frac n m}$ we obtain
\begin{align}
\Pr\left( S_m - \E S_m> ck\sqrt{\frac m n}\right)
&< e^{\frac{1}{2}+o(1)-c}, \qquad \textrm{ and } \qquad 
\Pr\left( S_m - \E S_m <  -ck\sqrt{\frac m n}\right)
&<e^{\frac{1}{2}+o(1)-c}. \nonumber
\end{align}
Thus for $n$ sufficiently large, the probability that $S_m$ diverges from its expectation by more than $ck\sqrt{\frac{m}{n}}$ is at most $2e^{(\frac{1}{2}+o(1)) -c}<4e^{-c}$.
\end{proof}

Equation~\ref{T1-TK} can also be formulated as 
\begin{align} S_{T(\mathbf{X})-1} < T(\mathbf{Y}) \leq S_{T(\mathbf{X})} \label{sandwich}.
\end{align}
Lemma~\ref{Sm-concentration} gives us that $\vert S_{m} - \E S_{m}\vert < \sqrt{mk}$  with probability $1- O(e^{-\sqrt{n/k}})=1-o(1)$. Since each $\ell_i$ is independent from $T(\mathbf{X})$ (how long it takes to collect one $k$-set tells us nothing about how many $k$-sets are needed to cover the entire set of coupons),  we can use the lemma with  $m=T(\mathbf{X})$ to bound the right-hand side of inequality (\ref{sandwich}), and $m=T(\mathbf{X})-1$ for the left-hand side. (The lemma requires that $T(\mathbf{X}) > \frac{n}{k}$, which holds whp by the first moment method.) This gives us that, whp,
\[ k(T(\mathbf{X})-1) -\sqrt{k(T(\mathbf{X})-1)}< T(\mathbf{Y}) \leq kT(\mathbf{X}) +\sqrt{kT(\mathbf{X})}.\]
By Theorem~\ref{theorem: erdos renyi} we have that $T(\mathbf{X})< \frac{2n \log n}{k} $ holds whp and that $\vert T(\mathbf{Y}) - n\log n\vert <cn$ holds with probability at least $1-e^{-c}+o(1)$. Applying the triangle inequality, we see that
\begin{align*}
 \left\vert T(\mathbf{X})-\frac{n\log n }{k}\right\vert &\leq   \left\vert \frac{T(\mathbf{Y})}{k}-\frac{n\log n }{k}\right\vert+\left\vert T(\mathbf{X})-\frac{T(\mathbf{Y})}{k}\right\vert \\
&\leq c\cdot \frac{n}{k}+\frac{\sqrt{T(\mathbf{X})}}{k}+1\leq c\cdot \frac{n}{k} + \frac{\sqrt{2n \log n}}{k\sqrt{k}} =  (c+o(1))\cdot \frac{n}{k}
 \end{align*} 
holds with probability at least $1-e^{-c}+o(1)$. The theorem follows.  
\end{proof}

We now turn our attention to Theorem~\ref{main thm: exchangeable coupons}. Suppose that we have an exchangeable random covering variable $W$ for the set $V=[n]$. Let $\mu=\mathbb{E}\vert W\vert$, let $M$ be the maximum value that $\vert W \vert $ takes with strictly positive probability, and let $\chi=\mathbb{E} [\vert W\vert^2]$.

Coupling the $W$-coupon sequence $\mathbf{W} = (W_1, W_2, \ldots)$ with the singleton coupon sequence $Y_1, Y_2, \ldots$ as in the proof of Theorem~\ref{thm: k-uniform exchangeable coupon collector, sublinear k}, we get the following analogue of Equation~(\ref{T1-TK}):
\begin{equation}
\sum_{i=1}^{T(\mathbf{W})-1} \ell_i < T(\mathbf{Y}) \leq \sum_{i=1}^{T(\mathbf{W})} \ell_i
\label{T1-TW},
\end{equation}
where $\ell_i$ is the least integer such that $C_{Y}[\ell_1+\cdots+\ell_{i-1}+1, \ell_{1}+\cdots+\ell_i]= \vert W_i\vert$. Applying Wald's inequality, we get that
\begin{align} \frac{\mathbb{E}T(\mathbf{Y})}{\E \ell_1}\leq \mathbb{E}T(\mathbf{W}) <1+\frac{\mathbb{E}T(\mathbf{Y})}{\E \ell_1}.\label{eq: Wald bound}\end{align}
In particular if $\mathbb{E}\ell_1=o(n\log n)$, we have $\mathbb{E}T(\mathbf{W})=(1+o(1))\frac{n\log n}{\E \ell_1}$. An inconvenient aspect of this expression is that it remains in terms of $\E \ell_1$, the expected number of single coupon we need to draw in order to see $\vert W\vert$ distinct coupons. However if $M=o(n)$, note that for any $m\leq M$ the expected number of single coupons we need to draw in order to see $m$ distinct coupons is
\begin{align}
\sum_{i=0}^{m-1}\frac{n}{n-i}=(1+o(1))n \log \left(\frac{n}{n-m}\right)= (1+o(1))m,\label{eq: number of coupons needed for m distinct}\end{align}
and thus $\E \ell_1=(1+o(1))\mathbb{E} \vert W\vert$. Together with (\ref{eq: Wald bound}), (\ref{eq: number of coupons needed for m distinct}) establishes the following:
\begin{proposition}
For the $W$-collector with maximum coupon size $M$ and mean coupon size $\mu$, the following hold: 

	\begin{enumerate}[(i)]
		\item if $M=o(n)$, $\mathbb{E}T(\mathbf{W})= (1+o(1))\frac{n \log n}{\mu}$;
		\item if $\E \ell_1=o(n\log n)$, $\mathbb{E}T(\mathbf{W})= (1+o(1))\frac{n \log n}{\E \ell_1}$;
		\item if $\E \ell_1=\Omega(n\log n)$, $\mathbb{E}T(\mathbf{W})=O(1)$. 
	\end{enumerate}
\end{proposition}\qed

\noindent Theorem~\ref{main thm: exchangeable coupons}, which we now prove gives conditions for the covering time $T(\mathbf{W})$ to be sharply concentrated around its expected value. 

\begin{proof}[Proof of Theorem~\ref{main thm: exchangeable coupons}] 	
	We first prove that if any of conditions \textit{(i)--(iii)} holds, then whp $T(\mathbf{W})=(1+o(1)) \frac{n\log n}{\mu}$. Note that  \textit{(i)--(iii)} give us $M=o(n)$, whence $\E \ell_1\leq M(1+o(1))=o(n\log n)$. As in Theorem~\ref{thm: k-uniform exchangeable coupon collector, sublinear k}, having sandwiched $T(\mathbf{Y})$ between two sums of independent identically distributed random variables $S_{T(\mathbf{W})-1}:=\sum_{i=1}^{T(\mathbf{W})-1}t_i$ and $S_{T(\mathbf{W})}:=\sum_{i=1}^{T(\mathbf{W})}t_i$, the crux of the proof lies in showing these two (random) sums are concentrated around their mean. Indeed, provided we can show that whp $S_{T(\mathbf{W})}=(1+o(1)) T(\mathbf{W}) \E \ell_1$ and $S_{T(\mathbf{W}) -1}=(1+o(1))(T(\mathbf{W})-1) \E \ell_1$, we have that whp
	\[(1+o(1))\frac{n\log n}{\mathbb{E} \ell_1}=(1+o(1))\frac{T(\mathbf{Y})}{\E \ell_1} \leq T(\mathbf{W}) \leq (1+o(1))\frac{T(\mathbf{Y})}{\mathbb{E} \ell_1}+1 = (1+o(1))\frac{n\log n}{\E \ell_1}\]
	by appealing to Theorem~\ref{theorem: erdos renyi} (and the fact that $\E \ell_1=o(n\log n)$ by (\ref{eq: number of coupons needed for m distinct}). Let us therefore establish the concentration we require.

	We use the following generalized Chernoff bound, see e.g. Theorems 2.8 and 2.9 in~\cite{ChungLu06}.
	\begin{lemma}[Generalized Chernoff bound]\label{lemma: chernoff bound}
		Let $(U_i)_{i=1}^t$ be a sequence of independent, identically distributed non-negative integer-valued random variables, with $U=U_1 \leq M$ with probability $1$. Let $\varepsilon>0$ be fixed. Then
		\[\Pr\left[\left\vert \sum_{i=1}^t U_i - t\mathbb{E}U\right\vert\geq \varepsilon \mathbb{E}U \right] \leq 2 e^{-\frac{\varepsilon^2 t^2 (\mathbb{E}U)^2} {2 t\mathbb{E}[U^2]+2Mt\mathbb{E}U /3}}.\]
	\end{lemma}
	\noindent We apply the Lemma to $\vert W\vert$. Suppose condition \textit{(ii)} holds. Then $M=o(n)$ and thus $\E \ell_1=(1+o(1))\mu$. For any fixed $\varepsilon>0$ and $t=(1+o(1)) \frac{n\log n}{\mu}$ we have that
	\begin{align*}
	Pr\left[\left\vert \sum_{i=1}^t \vert W_i\vert - t\mu\right\vert \geq \varepsilon t \mu \right] &\leq 2 e^{-\frac{\varepsilon^2 t \mu^2} {2 \chi +2M \mu/3}} %\\
	\leq 2e^{ -\varepsilon^2 (1+o(1))\frac{ n\log n \mu}{2 M^2+ 2M\mu /3}}=e^{-\varepsilon^2 (1+o(1))\frac{ n\log n \mu}{2M^2}}=o(1),
	\end{align*}
	where the last inequality used the fact that $M=o(n\log n)$.
	Thus for $t$ around the expected value of $T(\mathbf{W})$, the sum $S_t=\sum_{i=1}^{t}\ell_i$ is whp concentrated around its mean $(1+o(1))\mu t$. It follows that if \textit{(ii)} is satisfied then whp $T(\mathbf{W})= (1+o(1))\frac{T(\mathbf{Y})}{\mu}$, as desired. Since condition \textit{(ii)} implies \textit{(i)} this also establishes that \textit{(i)} is sufficient for $T(\mathbf{W})$ to be sharply concentrated around $\frac{n \log n}{\mu}$. For condition \textit{(iii)}, we use the same argument as for \textit{(ii)} but use the assumption $\chi= o(n\log n\mu)$ to bound $\chi$ instead of the bound $\chi \leq M^2$.

	For conditions \textit{(iv)}, we show that we can truncate $W$; for $\varepsilon>0$  fixed, Chebyshev's inequality implies
	\[\Pr[\bigl\vert \vert W\vert -\mu \bigr \vert> \varepsilon \mu ] \leq \frac{\chi -\mu^2}{\varepsilon^2 \mu^2}=o\left(\frac{1}{n\log n\mu}\right).\]
	Thus the expected number of coupons with size differing from $\mu$ by more than $\varepsilon \mu$ which occur by time $t=(1+o(1))\frac{n\log n}{\E \ell_1}\leq (1+o(1))\frac{n\log n}{\mu}$ is $o(1)$. By Markov's inequality whp no such coupon is seen by that time, and we can couple/sandwich the $W$-coupon collectors between two $k$-uniform exchangeable coupon collectors $X^-$ and $X^+$, collecting coupons of size $k_-=(1-\varepsilon)\mu$ and $k_+=(1+\varepsilon)\mu$ respectively, in such a way as to have $T(\mathbf{X^-})\leq T(\mathbf{W})\leq T(\mathbf{X^+})$.

	We then split into two cases. If $\mu=o(n)$, then by Theorem~\ref{thm: k-uniform exchangeable coupon collector, sublinear k} whp these two sandwiching coupon collectors finish at times $T(\mathbf{X^-})=(1+o(1))\frac{n\log n}{(1-\varepsilon) \mu}$ and $T(\mathbf{X^+})=(1+o(1))\frac{n\log n}{(1+\varepsilon) \mu}$ respectively. Since $\varepsilon>0$ was arbitrary we deduce that $T(\mathbf{W})=(1+o(1))\frac{n\log n}{\mu}$ as desired. If on the other hand $\mu=cn$ for some $c\in(0,1)$, then by Theorem~\ref{main thm: exchangeable uniform} whp these two sandwiching coupon collectors finish at times $T(\mathbf{X^-})=(1+o(1))\frac{\log n}{-\log \left(1-c(1-\varepsilon)\right)}$ and $T(\mathbf{X^+})=(1+o(1))\frac{\log n}{-\log \left(1-c(1+\varepsilon)\right)}$ respectively (provided we picked $\varepsilon$ sufficiently small so that $c(1+\varepsilon)<1$ and $c(1-\varepsilon)>0$). Since $\varepsilon>0$ was arbitrary we deduce that $T(\mathbf{W})=(1+o(1))\frac{\log n}{-\log(1-c)}$ as desired.
\end{proof}

%-----------------------------------------------------------------------------------------------------------------------------------------------------------------------------------------------------------------------------
\section{Fast coverage}\label{section: fast coverage}
Let $V$ be an $n$-set, and let $X$ be a random covering variable for $V$ with average coupon size $\mu=\mathbb{E} \vert X\vert$. If $\mu<(1-\delta)n$ for some fixed $\delta>0$ and  $X$ is exchangeable and uniform, then whp the covering time $T(\mathbf{x})$ for the $X$-coupon collector satisfies $ T(\mathbf{X})=(1+o(1))\frac{\log n}{-\log\left(1-\frac{\mu}{n}\right)}$ (Theorem~\ref{main thm: exchangeable uniform}). However if we replace the `exchangeable' assumption by `transitive', $T(\mathbf{X})$ can be sharply concentrated on a strictly smaller value.  For a balanced, not necessarily uniform $X$ with average coupon size $\mu<(1-\delta)n$, we say that the $X$-coupon collector  is \emph{fast} if there exists a strictly positive constant $\eta >0$ such that whp $T(\mathbf{X}) < (1-\eta )\frac{\log n} {\left(1-\frac{\mu}{n}\right)}$. In this section, we briefly discuss fast coverage. We have already seen one example of a fast coupon collector in Example~\ref{example: lottery}. We now give a second example of a fast collector which demonstrates a different way of getting fast coverage.

\begin{example}\label{example: collecting a smaller set}[Coupon collecting on a smaller set]
Let $V=[kn]$. For every $i\in\{1,2, \ldots k\}$, let  $X=\{(i-1)k+1, (i-1)k+2, \ldots ik\}$ with probability $\frac{1}{n}$.
\end{example}
The covering variable $X$ in the example above is transitive and $k$-uniform. Set $N=\vert V\vert$ and $k=n^{\alpha}=N^{\frac{\alpha}{1+\alpha}}$.  Provided $k=o(N)$ (i.e. provided $\alpha=O(1)$), the covering time of a exchangeable $k$-uniform coupon collector on an $N$-set is whp concentrated around $(1+o(1))\frac{N\log N}{k}$. 
However the $X$-coupon collector is really collecting from a smaller set of size $n$: we may identify each of the coupons $\{(i-1)k+1, (i-1)k+2, \ldots ik\}$ with a singleton $\{x_i\}$. We can then couple the $X$-collector on $[N]$ with a $1$-uniform exchangeable coupon collector $\mathbf{X'}$ on the set $\{x_1, x_2, \ldots x_n\}$. By Theorem~\ref{theorem: erdos renyi}, the covering time $T(\mathbf{X})$ is thus whp concentrated around $T(\mathbf{X'})=(1+o(1)) n\log n=\left(\frac{1}{1+\alpha}+o(1)\right)\frac{N\log N}{k}$. Thus for any $\alpha>0$, the $X$-coupon collector finishes collecting earlier than one would expect knowing only the mean-size of its coupons.

%-------------------------------------------------------------------------------------------------------

\subsection{Sufficient conditions for fast coverage}\label{subsection: conditions for fast coverage}
We have given two instances of fast coverage so far. In Example~\ref{example: lottery}, fast coverage occurred because though the average coupon size was small, there was a small chance of  `winning the lottery' and receiving a very large coupon. In Example~\ref{example: collecting a smaller set}, fast coverage occurred because $X$ was structured in such a way that the problem of covering $V=[kn]$ with $k$-sets was actually equivalent to the problem of covering a much smaller set $V'=[n]$, which could be achieved more rapidly (and also entailed having some very large pairwise correlations $q_{xy}$).

We can restate these two `speeding up' properties in a formal way. 
\begin{theorem}\label{theorem: sufficient conditions for fast coverage}
Let $V$ be an $n$-set. Let $X$ be a transitive coupon variable for $V$ with average coupon size $\mu=o(n)$. Then if any of the following conditions are satisfied, $X$ is fast:
\begin{enumerate}[(i)]
	\item there exist some $\varepsilon>0$ and $C\geq 1+ \varepsilon$ such that $\mathbb{P}[\vert X\vert\geq C\mu]\geq \frac{1+\varepsilon}{C}$; 
	\item  there exists $1\ll n'\leq \frac{n}{\mu}$, and a partition of $V$ into $n'$ subsets $V=\sqcup_{i=1}^{n'}V_i$ such that $\mathbb{P}(V_i \subseteq X )\geq (1+\varepsilon)\frac{ \log n'}{\log n}\left(-\log \left(1-\frac{\mu}{n}\right)\right)$ for every $i\in[n']$. 

\end{enumerate}
\end{theorem}
\begin{proof}
Suppose condition (i) is satisfied. Let $\eta>0$ be a fixed positive number to be fixed later. We say that coupons of size at least $C\mu$ are \emph{large}, and we call other coupons \emph{small}. We couple $X$ with a transitive $C\mu$-uniform covering variable $Y$, by setting $Y$ to be a  $C\mu$-subset of $X$ chosen uniformly at random if $X$ is large, and to be the empty set otherwise. By Proposition~\ref{proposition: elementary bounds}, whp the $Y$-collector will need at most $(1+\eta) \frac{\log n}{-\log \left(1- \frac{C\mu}{n}\right)}$ non-empty coupons to cover $V$. Set $p =\frac{1+\varepsilon}{C}$. Let $t$ be an integer with 
\[\left(\frac{1+\eta}{1-\eta}\right) \left(\frac{1}{p}\right) \frac{\log n}{- \log \left(1 -\frac{C\mu}{n}\right) } \leq t \leq \left( 1-\eta\right) \frac{\log n}{- \log\left( 1-\frac{\mu}{n}\right)}.\]
Since the left hand side is at most $\frac{1+\eta}{1-\eta}\frac{1+o(1)}{1+\varepsilon}\frac{\log n}{-\log\left(1-\frac{\mu}{n}\right)}$, picking $\eta$ sufficiently small relative to $\varepsilon$ and $n$ sufficiently large, we can always do this. We claim that whp the $Y$-collector will have covered all of $V$ by time $t$. Indeed, the probability that $Y\neq\emptyset$ is, by assumption, at least $p$. By a standard Chernoff bound, 
the probability that at least $(1-\eta)pt$ of the first $t$ coupons of the $Y$-coupon collectors are non-empty is at least $1- e^{-\frac{\eta^2 pt}{3}}=1-o(1)$. (Here we use the fact that $pt=\Omega\left(\frac{\log n}{-\log\left(1-\frac{C\mu}{n}\right)}\right)\rightarrow \infty$ as $n\rightarrow \infty$.) Thus whp by time $t$ we have seen at least $(1-\eta)pt$ non-empty $Y$-coupons; since, by our choice of $t$, this is at least $(1+\eta)\frac{\log n}{ - \log \left(1 -\frac{C\mu}{n}\right) }$, whence whp these non-empty $Y$-coupons cover all of $V$. Our coupling of $Y$ with $X$ then implies that whp $T(\mathbf{X})\leq t$. Since by definition $t\leq \left( 1-\eta\right) \frac{\log n}{- \log\left( 1-\frac{\mu}{n}\right)}$, we conclude that $X$ is fast.

For the second part of the theorem, suppose condition (ii) is satisfied. We define a random covering variable $Z$ for $[n']$ as follows: set $Y=\{i: \ V_i \subseteq X\}$. Set $p =(1+\varepsilon)\frac{\log n'}{ \log n}\left(-\log \left(1-\frac{\mu}{n}\right)\right)$. Let $\eta>0$ be chosen sufficiently small so that $1+\varepsilon> \frac{1+\eta}{1-\eta}$. Let $t$ be an integer with 
\[ \frac{(1+\eta)\log n'}{p} \leq t \leq (1-\eta) \frac{\log n}{-\log \left(1-\frac{\mu}{n}\right)} .\]
By our choice of $\eta$, and for $n$ sufficiently large, we can always pick such a $t$. We claim that whp the $Y$-collector will have covered all of $[n']$ by time $t$. Indeed by condition (ii) the expected number of $i\in[n']$ not covered by the $Y$-coupon collector by time $t$ is
\begin{align*}
\sum_{i\in [n']} (1-\mathbb{P}(i\in Y))^t&\leq n' (1-p)^t\leq e^{-\eta \log (n')}=o(1),
\end{align*}
so that by Markov's inequality whp the $Y$-coupon collector has covered $[n']$ by time $t$. By the coupling of $Y$ with $X$, and the fact that $\bigcup_i V_i =V$, it follows that whp $T(\mathbf{X})\leq t$. Since we chose $t \leq (1-\eta) \frac{\log n}{-\log \left(1-\frac{\mu}{n}\right)}$, we conclude that $X$ is fast.
\end{proof}

Theorem~\ref{theorem: sufficient conditions for fast coverage} leaves a number of interesting questions open. To begin with, are there other, subtler ways of being fast than either winning the lottery or collecting a smaller coupon set? In particular, are there conditions on the pairwise intensities $(q_{xy})_{x,y \in V}$ which imply fast coverage? Furthermore, Theorem~\ref{theorem: sufficient conditions for fast coverage} says nothing on what the probable value of $T(\mathbf{X})$ actually is. In cases where $X$ is fast, can we determine good bounds for $\mathbb{E}T$? With its ties to the $k$-SAT problem (see the next section), this is one of the most important open problems related to this paper.

%-----------------------------------------------------------------------------------------------------------------------------------------------------------------------------------------------------------------------------
\section{Applications}\label{section: applications}

%-------------------------------------------------------------------------------------------------------
\subsection{Connectivity in random graphs}
We consider the discrete time multigraph process $(G_t)_{t\geq 0}$ obtained by starting with the empty graph $G_0$ on $V=[n]$ and at each time step $t\geq 1$ selecting an edge $uv$ uniformly at random and adding it to $G_{t-1}$ to form $G_t$. We associate $n/2$ coupon collectors $\mathbf{X}^i$ to this process, $1\leq i \leq \frac{n}{2}$. The $i^{\textrm{th}}$ such collector aims to cover each $i$-set $A$ with an edge from $A$ to $V\setminus A$. Since each edge $uv$ connects $2\binom{n-2}{i-1}$ $i$-sets to their complements in $V$, the $i^{\textrm{th}}$ collector is $2\binom{n-2}{i-1}$-uniform and balanced, and aims to cover a set of size $\binom{n}{i}$. By Proposition~\ref{proposition: elementary bounds}, we thus have that her covering time $T(\mathbf{X}^i)$ will be whp at most $(1+o(1))t_i$ where 
\begin{align*} 
t_i &= \frac{\log \binom{n}{i}} {-\log\left(1- \frac{2\binom{n-2}{i-1}}{\binom{n}{i}}\right)}= \frac{\log \binom{n}{i}}{-\log\left(1- \frac{i(n-i)}{\binom{n}{2}}\right)}.
\end{align*}
For $i=o(n)$, $t_i=t_1 - \frac{n\log i}{2} +o(n)$, while for $i=\theta(n)$ $t_i=O(t_1/\log n)$. Further by Proposition~\ref{proposition: elementary bounds} we know that for any fixed $\eta>0$  we have that $T(\mathbf{X}^i)> (1+\eta)t_i$ with probability at most $n^{-\eta}$. Also in the case $i=1$ the collector's random coupon variable is in fact exchangeable and $2$-uniform. By Theorem~\ref{thm: k-uniform exchangeable coupon collector, sublinear k}, for any $x> 0$
\begin{align*}
\Pr(T(\mathbf{X}^1)>t_1+\frac{xn}{2})&\leq e^{-x}(1+o(1)) &\quad{ and } \quad \Pr(T(\mathbf{X}^1)<t_1-\frac{xn}{2})&\leq e^{-x}(1+o(1)).
\end{align*}

Thus by the union bound we have that for any $x=x(n)>0$, 
\begin{align*}\Pr\left(\max_i  T(\mathbf{X}^i)>t_1+ \frac{xn}{2}\right)&\leq \sum_i \Pr\left(T(\mathbf{X}^i)> t_i\left(1+\frac{\log i +x}{\log n}\right)(1+o(1))\right)\\
&\leq (1+o(1))\sum_{i\geq 1} e^{-\log i +x} \leq (1+o(1))e^{-x}\log n  
\end{align*}

In particular, setting $x=\varepsilon \log n$, the inequality above together with our bound on $\mathbb{P}(T(\mathbf{X}^1)<t_1-xn)$ establishes the following:
\begin{theorem}\label{thm: connectivity of Gt} Let $\varepsilon>0$ be fixed. Then
\begin{align*} \Pr\left(G_t \textrm{ is connected} \right)&\leq n^{-\varepsilon+o(1)} \qquad& \textrm{for }t\leq \frac{n \log n}{2}(1-\varepsilon),\\
\Pr\left(G_t \textrm{ is connected} \right)&\geq 1-n^{-\varepsilon +o(1)} \qquad & \textrm{for }t\geq \frac{n \log n}{2}(1+\varepsilon).\end{align*}
\end{theorem} \qed

It is easy to relate $G_t$ to the \emph{size model} $G_{n,m}$ of random graphs obtained by selecting $m$-distinct edges uniformly at random and adding them to the empty graph on $n$ vertices. Indeed Markov's inequality shows that for $t=O(n\log n)$, whp $G_t$ contains only $O(\frac{t^2}{n^2})=O((\log n)^2)$ repeated edges, so one can couple $G_t$ with with $G_{n,m}$ up to the connectivity threshold for $G_t$ in such a way that $G_{n, t-O((\log n)^2)} \subseteq G_t \subseteq G_{n, t}$. In this way, Theorem~\ref{thm: connectivity of Gt} above allows us to recover (a slightly weaker form of) the classical results of Erd{\H o}s and R\'enyi~\cite{ErdosRenyi1960} on the connectivity threshold for $G_{n,m}$: 
whp $G_{n,m}$ becomes connected at size $m=(1+o(1))\frac{n \log n}{2}$.

%-------------------------------------------------------------------------------------------------------
\subsection{Covering a square with random discs}
We return to Example~\ref{example: covering a square with discs}. Let $V$ be the torus obtained by identifying the opposite sides of the square of area $n$ $[0, \sqrt{n}]^2\subset \mathbb{R}^2$, and let $X$ be the intersection of $V$ with the disc of radius $r=r(n)$ about a uniformly chosen random point $x\in V$. Draw a sequence $\mathbf{X} = (X_1, X_2, \ldots )$ of independent random subsets of $V$ distributed according to $X$. When does their union whp cover $V$? This is known as a \emph{coverage} problem, and is a continuous analogue of the coupon collector problem. Coverage problems have been widely studied in random geometric graph theory, with motivation coming from applications to wireless networks, especially sensor networks (see the introduction of~\cite{HaenggiSarkar13} for a history of coverage problems).

We discretise the problem and apply our results to show sharp concentration of the covering time $T=T(\mathbf{X})$ in the case where $r(n)$ is of order $o(\sqrt{n})$ and bounded away from $0$ (so the measure of $X$ is $O(\pi r^2)=o(n)$). 
Tile $V$ with squares of side length $s$, where $s=s(r,n)$ is chosen so that $s=o(r)$ and $\sqrt{n}/s \in \mathbb{N}$. Let $\mathcal{T}$ denote the collection of all the tiles; by construction, $\vert \mathcal{T}\vert =n/s^2$ Given a disc $D$ of radius $r$ in $V$, we let $I_{-}$ to be the collection of tiles wholly contained inside $D$, and $I_+$ to be the collection of tiles having non-empty intersection with $D$. The random variable $X$ gives rise, via $I_-$ and $I_+$, to two random variables $X_-$ and $X_+$ taking values among the subsets of $\mathcal{T}$.

For any $D$ as above, it is easy to show (see e.g. Lemma~8 of \cite{FalgasRavryWalters12}) that the boundary of $D$ meets at most $\frac{18\pi r} {s}$ tiles; thus $\vert I_-\vert$ and $\vert I \vert$ are both within $\frac{18 \pi r}{s}$ of $\frac{\vert D\vert}{s^2}=\frac{\pi r^2 }{s^2}$. 
 Both of $X_-$ and $X_+$ are clearly balanced random covering variables for $\mathcal{T}$.

By Theorem~\ref{main thm: balanced coupons} their covering times $T(\mathbf{X}_-)$ and $T(\mathbf{X}_+)$ are therefore whp concentrated around $\frac{\log (ns^{-2})}{-\log \left(1- \frac{\pi r^2}{n}\right)}=(1+o(1))\frac{n\log n}{\pi r^2}$. Since by construction of the random variable $X_-$ and $X_+$ we have that $T(\mathbf{X}_-)\leq T(\mathbf{X}) \leq T(\mathbf{X}_+)$, we deduce that whp the covering time for the torus $V$ satisfies $T(\mathbf{X})=(1+o(1))\frac{n\log n}{\pi r^2}$.

It is easy to adapt the argument above to show that the covering time does not change significantly if instead of a torus we try to cover a square $S$ of area $n$ with discs of radius $r$ centred at uniformly chosen random points in $S$. The random covering variables we use are no longer quite balanced: there are $O(\frac{r\sqrt{n}}{s^2})$ tiles within distance $r$ of the boundary of $S$, each of which is covered with probability at least $\frac{\pi r^2}{2n}(1+o(1))$, and $O(\frac{r^2}{s^2})$ tiles within distance $r$ of a corner of $S$, each of which is covered with probability at least $\frac{\pi r^2}{4n}(1+o(1))$. The first moment method shows both of these sets of `boundary tiles' are whp covered by the time we have drawn $(1+\varepsilon)\frac{n \log n}{\pi r^2}$ discs, while the `central tiles' at distance at least $r$ from the boundary are whp covered by that time by our result for the torus. This yields the following well-known result on covering processes (see~\cite{Hall88}).
\begin{theorem}\label{thm: covering squares or discs}
Let $V$ be a square or torus of area $n$. Let $X$ be the intersection of $V$ with a disc of radius $r$ about a uniformly chosen random point in $V$, where $r=r(n)$ is bounded away from $0$ and satisfied $r(n)=o(\sqrt{n})$. Then whp the covering time $T$ of the continuous $X$-coupon collector on $V$ satisfies $T(\mathbf{X})=(1+o(1))\frac{n}{\pi r^2}$.  
\end{theorem}

More generally, our argument in the torus adapts immediately to any balanced random covering variable $X$ taking values among the compact subsets of $V$ and satisfying with probability $1$ (i) $\vert X\vert =o (\vert V\vert)$, and (ii)  $\vert \partial X \vert =o (\vert X\vert)$, where $\vert \delta X\vert$ denotes the measure (length) of the boundary of $X$. For such $X$, we again have
\[T(\mathbf{X}) =(1+o(1))\frac{\vert V\vert \log \vert V\Vert }{\mathbb{E} \vert X\Vert }.\]
Thus we may replace `disc' in the results above by e.g. `ellipse', `annulus', `square', `polygon', or even let $X$ be given by a probability distribution on a finite collection of shapes having the same Lebesgue measure and satisfying the required isoperimetric inequality (ii). These are special cases of a celebrated result of Janson~\cite{Janson86}.

%-------------------------------------------------------------------------------------------------------
\subsection{Covering the edges of a graph by spanning trees, and matroids by bases}
Let $G$ be a connected edge-transitive graph, on $n$ vertices, of minimum degree $d$ and let $X$ be a spanning tree of $G$ drawn uniformly at random from the set of all such trees. Our goal is now to cover the edge set $E$ of $G$ with the edges of trees from $X$.

It is well known that the random spanning tree is pairwise negatively correlated, with respect to the edges,  in fact it satisfies the even stronger negative correlation property of being a Rayleigh measure on $E$, see~\cite{BBL}. 

So, from Theorem~\ref{main thm: balanced coupons} we can conclude that the covering time $T$ is sharply concentrated around
$ \frac{(nd/2)  \log(nd/2)}{n-1}  $, as long as $d\gg 1$.

Covering the edge set of a graph is a special case of the problem of covering the ground set of a matroid by random drawn bases of the matroid.  In \cite{FM} it was shown that a large class of matroids, the \emph{balanced} matroids, which contain the class of cycle matroids of a graph, have pairwise negative correlation.  In the same way as for trees we can conclude that if a balanced matroid of size $n$ has rank $r$ then the covering time $T$ is sharply concentrated around $n \log(n)/r$, as long as $\log(r)=o(\log(n))$.

%-------------------------------------------------------------------------------------------------------
\subsection{Random $k$-SAT}\label{subsection: applications to SAT}
The Random Boolean Satisfiability  (SAT) problem is the following. Given $n$ boolean variables $x_1, x_2, \ldots x_n$

and an integer sequence $k=k(n)$, we form a random clause $C=l_1\vee l_2 \vee \ldots \vee l_k$ by selecting a $k$-subset $\{y_1,y_2, \ldots y_k\}$ of literals uniformly at random, setting $l_i=y_i$ with probability $1/2$ and $l_i =\lnot y_i$ otherwise, independently for each $i$, and taking $C$ to be the join of the literals $l_i$. We now consider a sequence of independent, identically distributed random clauses $C_1, C_2, \ldots $, with distribution given by $C$, and define a sequence of logical formulae in conjunctive normal form $F_t= \bigwedge_{i=1}^t C_i$ for $t=0,1, \ldots$. For $n\rightarrow \infty$, the random $k$-SAT problem asks whether or not there exists whp an assignment of truth values to the variables $x_1, \ldots, x_n$ such that the logical formula $F_t$ is satisfied. The random $k$-SAT problem is of fundamental importance to theoretical computer science and has been extensively studied (see~\cite{C-O}).

Here we note that this problem is equivalent to determining the covering time of a coupon collector problem. The space of satisfying assignments for a formula consisting of $t$ clauses involving $n$ variables can be viewed as the complement of the union of $t$ subcubes of $\{0,1\}^n$. If each of those $t$ clauses involves exactly $k$ distinct literals (that is, if we are working with an instance of $k$-SAT), then each of those $t$ subcubes has dimension $n-k$. In particular, we can couple the sequence of iid clauses $C_1, C_2, \ldots$ with a sequence of independent coupons $X_1, X_2, \ldots $, with $X_i \sim X$, where $X$ is the random coupon given by selecting an $(n-k)$-dimensional subcube of the $n$-dimensional discrete hypercube $V=\{0,1\}^n$ uniformly at random. The formula $F_t$ is then satisfiable if and only if the $X$-coupon collector has failed to cover $V$ by time $t$.

The random variable $X$ is $2^{n-k}$-uniform and transitive. Proposition~\ref{proposition: elementary bounds} thus gives some elementary upper bounds on the satisfiability threshold $T(\mathbf{X})$ for $F_t$: for any $\varepsilon>0$, whp
\[T(\mathbf{X}) \leq (1+\varepsilon)\frac{\log 2^n }{-\log (1 -2^{-k})}= (1+\varepsilon) n \frac{\log 2}{-\log(1-2^{-k})} .\]

For $k(n)$ large enough this bound is in fact an equality, as first proven in \cite{FW}. Using Theorem \ref{main thm: balanced coupons}  we can obtain the same result.
\begin{theorem}
	Let $k=\log_2 n + \omega(n)$, where $\omega(n)\to \infty $, then whp $T(\mathbf{X})= (1+o(1)) n 2^k \log 2$
\end{theorem}
\begin{proof}
	Let $N=2^n$ be the number of vertices in the hypercube $Q_n$, our base set. We shall show condition (i) in  Theorem~\ref{main thm: balanced coupons} is satisfied to deduce the claimed sharp concentration result for $T(\mathbf{X})$. Checking that (i) holds is a matter of simple computations.  Most of the estimates needed here are standard so we only sketch  the argument.

	We note first of all that in our setting, for any pair of vertices $x$ and $y$ at Hamming distance $i$ in the hypercube $Q_n$,  \[q_{xy}=\frac{{n-i \choose k}}{2^k{n \choose k}}\] By symmetry, condition (i) is equivalent to
	\[S=\sum_{y\neq 0}  (\exp(tq_{0,y})-1) = o(N).\]

	Now the threshold for $k$-satisfiability we shall obtain from Theorem~\ref{main thm: balanced coupons} (which is the first moment threshold) is $N\log N/ (N/2^k)=2^kn \log 2$. We therefore let $t = 2^kn\log 2  \cdot(1- \delta_n)$ for some 
	$\delta_n = o(1)$ to be determined later.
	Now
	\begin{align*}
	S
	&=\sum_{y}  [-1+ \exp(n\log 2  \cdot(1- \delta_n) \cdot 2^k q_{0,y})]
	\\
	&=\sum_{i=1}^{n-k} {n \choose i } \left[-1+ \exp \Big(n\log 2  \cdot(1- \delta_n)  \cdot \frac{\binom{n-i}{k}}{\binom{n}{k}} \Big) \right]
	\end{align*}
		
	Let $a_i$ be the $i$:th term of this sum.
	We deal separately with the three cases  $i\geq\frac{n}{2} - \frac{n}{k}$, $\frac{n\ln k}{k-1}\leq i<\frac{n}{2} - \frac{n}{k}$ and $i<\frac{n\ln k}{k-1}$. In the first case, 
	we change the summation index so that $i = \frac{n}{2}-j$. Note that
	\[
		\frac{{n-i \choose k}}{{n \choose k}} \leq \big(1-\frac{i}{n}\big)^k = 2^{-k}\big(1+\frac{2j}{n}\big)^k \leq 2^{-k}\exp\Big(\frac{2jk}{n}\Big)
	\]
	
	Summing over all $j$ such that $-\frac{n}{2}\leq j<\frac{n}{k}$ we get that
	\begin{align*}
	2^{-n}\sum_{i= \frac{n}{2}-\frac{n}{k}}^{n-k} a_i
	&=2^{-n}\sum_{j=k-\frac{n}{2}}^{\frac{n}{k}} {n \choose \frac{n}{2}-j } \left[-1+ \exp \Big(n\log(2)  \cdot \frac{{\frac{n}{2}-j \choose k}}{{n \choose k}} \Big) \right]
	\\
	&\leq -1+\exp(2^{-\omega(n)}e^{2}) = o(1)
\end{align*}

	In the second case a convexity argument shows that 
	\[
	2^{-n} \sum_{i=\frac{n \ln k}{k-1}}^{\frac{n}{2}-\frac{n}{k}} a_i
	\leq \exp\Big(\log n -\frac{2n}{k^2}+o(1)  \Big)
	= o(1)
	\]	
	
	Finally, for $i \leq \frac{n\log k}{k-1}$, coarser bounds suffice: ${\binom{n-i}{k}/\binom{n}{k} \leq 1}$ and $\log \binom{n}{i} \leq 2i \log(n/i)$. Thus
	\[
	2^{-n} \sum_{i=1}^{\frac{n \log k}{k-1}} a_i
	\leq 2^{-n} \sum_{i=1}^{\frac{n \log k}{k-1}} \exp\Big(2i \log\big(\frac{n}{i}\big)+ n \log(2)(1-\delta_n) \Big)
	= n\exp\Big(n\Big[\frac{2\log(k)^2}{k}-\log(2)\delta_n\Big]\Big),
	\]
	which is $o(1)$ provided $\log(2) n\delta_n - \frac{2n\log(k)^2}{k} - \log(n) \to \infty$; this is satisfied for instance if we choose $\delta_n = n^{-\frac{1}{2}}$.

	Together these three cases, and the choice of $\delta_n$ above give that $2^{-n}\sum_{i=1}^{n-k} a_i = o(1)$, or in other words that $S=o(N)$, and condition (i) is satisfied (since $S=o(N)$ and $\sum_x e^{-q_x t} =e^{N\delta_n}=\omega(1)$). The result is then immediate from Theorem~\ref{main thm: balanced coupons}.
\end{proof}

For constant $k$ the simple first moment bound does not give the correct value for the satisfiability threshold. For $k=3$ our  simple upper bound is $T(\mathbf{X})\leq\left(5.190\ldots +o(1)\right)n$ and it has been shown that  $T(\mathbf{X})\leq 4.506n$, \cite{DBM}. Heuristics based on spin-glass theory has lead to the conjecture that the correct threshold is  $4.267\ldots n$, see \cite{MZ}. The well known  satisfiability conjecture states that for each $k$ there exists a constant $c_k$ such that the threshold for random $k$-SAT is $c_k n$. Recently a  proof of this conjecture for sufficiently large values of $k$ has been announced \cite{DSS}. It is also known \cite{C-O} that as $k$ increases the threshold location scales as $2^k \ln 2-1/2 (1 + \ln 2) + o_k(1)$, thus matching to leading order the bound given by the coupon collector.

%-------------------------------------------------------------------------------------------------------
\section{Concluding remarks}
Another natural coupon collector problem is the $q$-colourability of the uniform random graph. Here the set $V$ we are covering is the set of all strings of length $n$ over the alphabet $[q]$. Each string is interpreted as a vertex colouring of an $n$ vertex graph.  For each edge $e$ in the 
complete graph on $n$ vertices we create a coupon consisting of all colourings in which the endpoints of $e$ have the same colour.  The covering time $T$ for this coupon process now corresponds to the threshold for a uniform random graph of size $T$ on $n$ vertices ceasing to be $q$-colourable.

This coupon collector process is not balanced, but the covering time is essentially determined by the covering time of the almost balanced coupons so one can restrict to that 
subcase without loss of generality. Denote by $X$ the random coupon variable associated with the process; $X$ has size $q^{-1} \vert V\vert $ 
and is transitive and uniform, but is very much non-exchangeable: there are both strong positive and strong negative correlations between the 
various colourings, so that our Theorems~\ref{main thm: balanced coupons}, ~\ref{main thm:  balanced k-uniform} do not apply. For $q=2$, it is 
known that the covering time $T(\mathbf{X})$ is \emph{not} sharply concentrated. This stands in contrast with the situation for $q\geq 3$: in \cite{AN05} it was proven that the chromatic number of a random graph with edge probability $p=\frac{c}{n}$ has two possible values, and for 
all but a discrete sequence of values for $c$  whp only one value. This result would follow directly from a sharp threshold result for 
the coupon collector process described above.

A natural question is then whether any transitive, $cn$-uniform random  coupon variable $X$ with $c>0$ sufficiently small has sharp concentration of $T(\mathbf{X})$, i.e. whether random $q$-colouring threshold for large $q$ is determined by general coupon collector results 
(as opposed to specific structural features of the random colouring setting).

In a different direction, much remains to be done on the case of fast coupon collectors, as remarked at the end of Section~\ref{section: fast coverage}. The $k$-SAT problem for small $k$ gives us an example of a transitive, uniform and linear-sized coupon collector which is fast. The difficulty of that problem suggests the rigorous study of fast coupon collectors will be hard in general. Nevertheless we feel that the following problems are well-motivated, and for $\mu$ small enough may prove tractable. 
\begin{problem}
Let $X$ be a $\mu$-uniform transitive random covering variable for an $n$-set $V$
\begin{enumerate}
\item Give estimates for the value of $T_{\frac{1}{2}}$ in terms of $\mu$ and the pairwise intensities $q_{xy}=\Pr(\{x,y\} \subseteq X)$, $x,y \in V$;
\item Give sufficient conditions for $T(\mathbf{X})$ to be sharply concentrated about $T_{\frac{1}{2}}$. 
\end{enumerate}
\end{problem}

%-------------------------------------------------------------------------------------------------------

\end{document}